%% file: main.tex
\RequirePackage{fix-cm}
\documentclass[smallextended]{svjour3}       
\smartqed  
\usepackage{graphicx}
\usepackage[utf8]{inputenc}
\usepackage{amsmath,amssymb}
\usepackage{algorithm,algorithmic,ulem}
\usepackage{makecell}
\usepackage{graphicx,subcaption}
\usepackage{lineno}
\usepackage[inline]{enumitem}
\usepackage[dvipsnames]{xcolor}
\usepackage{comment}
\usepackage{natbib}
 
\usepackage{cleveref}

\input{notation}

\journalname{ArXiv preprint}
\begin{document}

\title{Randomized Preconditioned Solvers for Strong Constraint 4D-Var Data Assimilation
}

\titlerunning{Randomized Preconditioned Solvers for SC-4DVAR}        

\author{Amit N. Subrahmanya \and
        Vishwas Rao \and
        Arvind K. Saibaba
}


\institute{
            Amit N. Subrahmanya \at
            Computational Science Laboratory, Department of Computer Science, Virginia Tech, 620 Drillfield Dr., Blacksburg, 24061, VA, USA \\
            \email{amitns@vt.edu}           
            \and
            Vishwas Rao \at
            Argonne National Laboratory, 9700 S Cass Ave., Lemont, 60439, IL, USA\\
            \email{vhebbur@anl.gov}
            \and
            Arvind K. Saibaba \at
            Department of Mathematics, North Carolina State University, 2311 Stinson Dr., Raleigh, 27695, NC, USA\\
            \email{asaibab@ncsu.edu}
}

\date{}

\maketitle

\begin{abstract}

The Strong Constraint 4D Variational (SC-4DVAR) data assimilation method is widely used in climate and weather applications. 
SC-4DVAR involves solving a minimization problem to compute the maximum a posteriori estimate, which we tackle using the Gauss-Newton method. 
The computation of the descent direction is expensive since it involves the solution of a large-scale and potentially ill-conditioned linear system, solved using the preconditioned conjugate gradient (PCG) method. 
To address this cost, we efficiently construct scalable preconditioners using three different randomization techniques, which all rely on a certain low-rank structure involving the Gauss-Newton Hessian. 
The proposed techniques come with theoretical guarantees on the condition number, and at the same time, are amenable to parallelization.
We also develop an adaptive approach to estimate the sketch size and choose between the reuse or recomputation of the preconditioner. 
We demonstrate the performance and effectiveness of our methodology on two representative model problems---the Burgers and barotropic vorticity equation---showing a drastic reduction in both the number of PCG iterations and the number of Gauss-Newton Hessian products after including the preconditioner construction cost.





\end{abstract}

\section{Introduction}
\label{sec:introduction}

Data assimilation (DA), having origins in meteorology and oceanography, aims to optimally blend prior knowledge coming from a computational model of a dynamical system with noisy observations of reality to estimate the system states, parameters, or controls while quantifying the uncertainties of the said estimates~\cite[]{evensen2022data,asch2016book}.
Data assimilation can be viewed in a Bayesian framework where the current best estimate is obtained from a posterior distribution that combines a prior distribution (of the previous best estimate) with the observational likelihood~\cite[]{Kalnay_B2003}.
The Strong Constraint 4D-Variational approach (SC-4DVAR), is a well-known DA framework used to estimate the initial condition of a system under a perfect model assumption. 
Essentially, the SC-4DVAR discovers an optimal estimate of the initial condition that maximizes the Bayesian posterior probability over a spatiotemporal (hence, the name, 4D) trajectory consisting of background, observation, and model information~\cite[]{LeDimet_1986}.
Typically, the SC-4DVAR optimization problem~\cite[]{evensen2022data} is solved using Newton-based methods such as the Gauss-Newton (GN) approach~\cite[]{nocedal1999numerical,chen2011gnhess}.
Problems with fine-scale discretization require the solution of a large linear system to compute the GN descent direction. 
Iterative linear solvers are typically used since they require matrix-vector products rather than matrix decompositions, making them memory efficient.
In this work, we use the preconditioned conjugate gradient (PCG) method~\cite[]{saad2003iterative}, requiring the actions of {the forward model (FWD)}, tangent linear model (TLM), and the adjoint model (ADJ)~\cite[]{Kalnay_B2003,evensen2022data}. 
The performance of PCG depends crucially on the use of a good preconditioner, which is the central focus of this paper.

In recent years, randomized methods have emerged as a powerful tool in scientific computing for accelerating expensive computations (see, e.g., recent surveys such as~\cite{halko2011finding,martinsson2020randomized}). 
Randomized methods have strong theoretical guarantees, reduced computational costs, and are highly parallelizable. 
A key ingredient in the randomized methods is the use of a sketch of the matrix of interest (obtained by computing the product of the matrix with an appropriately chosen random matrix) which captures the desired essential features of the matrix. 
Randomization for least-squares problems roughly comes in two flavors: sketch and solve, or sketch and precondition. 
In the sketch and solve approach, randomization is used as a dimensionality reduction tool to solve a smaller inexact problem, whereas in the sketch and precondition approach, randomization is used to accelerate the convergence of the PCG method. 

\paragraph{Contributions} This paper proposes multiple, efficient, randomized solvers for SC-4DVAR. 
Our contributions are subdivided into three parts---algorithms, analysis, and numerical experiments. 
\begin{enumerate}
    \item \textbf{Algorithms}: Previously analyzed in ~\cite{flath2011fast,Tanbui_2013,spantini2015optimal,cui2014likelihood}, it is well established that the eigenvalues of the prior preconditioned data-misfit term of the GN Hessian exhibit rapid decay.
    We exploit this fact to propose three different randomized algorithms (i) Randomized SVD (RandSVD), (ii) Nystr\"om, and (iii) Single View (SingleView), to compute a low-rank sketch of the GN Hessian for SC-4DVAR, and demonstrate the reliability of its low-rank sketch, both in theory and practice.
    These three sketching methods can be used in the following ways: 
    \begin{enumerate}
        \item \textbf{SketchSolv}: The GN descent direction is obtained by solving a linear system with the sketched GN Hessian (using the Woodbury formula) to obtain a descent direction. 
        \item \textbf{SketchPrec}: The GN descent direction is obtained by using the GN Hessian in PCG preconditioned by the inverse of the low-rank sketch.  
        We also develop an adaptive approach (called \textbf{SketchPrecA}) that decides on both (1) the sketch reuse (rather than recompute) and (2) the sketch size of the low-rank approximation. 
    \end{enumerate}
    The first two randomized methods (RandSVD, Nystr\"om) are nearly identical in cost, error, and performance. 
    The third method (SingleView), while more expensive to compute, offers an additional level of parallelism.
    \item \textbf{Analysis}: A twofold analysis---{\em structural}, and {\em expectational}---of the proposed sketching methods is presented in this paper.
    \begin{enumerate}
        \item \textbf{Structural analysis}: We analyze the performance of the approximation for a single step of the GN iteration in {two metrics}: the error in the solution, and the condition number of the preconditioned operator of the GN iterations. 
        The analysis is applicable beyond low-rank approximations and the randomized algorithms proposed in this paper. 

        \item \textbf{Analysis in expectation}: We develop specific analyses for the three randomized methods in expectation. 
        The bounds do not show explicit dependence on the size of the matrices and provide insight into the sampling parameters.
    \end{enumerate}
  
    \item \textbf{Numerical Experiments}: {We validate the proposed methods on two model problems: 1D Burgers equation, and the more challenging barotropic-vorticity equation.}
    We compare our methods with the state-of-the-art Lanczos approach, which is inherently sequential.
\end{enumerate}

When compared with the solver preconditioned by the prior (Prec\_{$\B{\gampr}$}), our experiments show that the randomized sketching methods result in
\begin{enumerate*}[label={(\roman*)}]
     \item faster convergence of PCG measured in the number of iterations,
     \item reduction in the total number of TLM and ADJ evaluations {in the online phase}, at the same time,
     \item enable parallel evaluation of multiple batches of TLM and ADJ vector products {in the offline phase}. 
 \end{enumerate*}
As demonstrated by the numerical experiments for the barotropic vorticity problem, the adaptive randomized sketching methods demonstrate even better performance. That is, even including the offline phase, the number of TLM and ADJ evaluations by the adaptive randomized sketching methods is $\sim$ 30\% -- 70\% less than that of Prec\_{$\B{\gampr}$}. Note that the offline phase lends itself to embarrassingly parallel evaluations.

For the Burgers equation, we also investigate 
\begin{enumerate*}[label={(\roman*)}]
    \item the number of random vectors necessary to obtain a good preconditioner,
    \item the influence of diffusion and problem mesh size on the conditioning of the GN Hessian matrix and its relative spectral decay, and
    \item the effect of changing the mesh size on PCG iterations.
\end{enumerate*}
The recommendation is to use the adaptive preconditioning approach---SketchPrecA---due to lower computational cost (uses fewer TLM and ADJ evaluations than other SketchPrec methods).  
At the same time, SketchSolv is not preferred over SketchPrec due to its inherent inexactness. 

\paragraph{Literature review}
 
Randomization has applications in accelerating computations in Bayesian inverse problems in several ways. 
Firstly, it has been used to accelerate optimization by using randomized approximations of the cost function, such as the stochastic and sample average approximation methods in \cite{Shapiro2009}.
Secondly, randomization is used in inverse problems to approximate the posterior covariance through a low-rank approximation of the prior-preconditioned data-misfit part of the Hessian, see~\cite{villa2021hippylib,saibaba2015fast}. 
Third, in~\cite{saibaba2021randomized}, it has been used to reduce the number of PDE solves in Bayesian inference. 
The randomized approaches used here also target the optimization problem of computing the MAP estimate, but use the same idea of low-rank approximation used to approximate the posterior covariance.  

There is a long history of preconditioning iterative methods in the context of data assimilation, see~\cite{freitag2020numerical} for some key references. 
\cite{Tshimanga_2008} present spectral, quasi-Newton, and Ritz limited memory preconditioners for hastening PCG convergence for the incremental variant of the SC-4DVAR (essentially treating the non-linear optimization as a sequence of quadratic optimizations).
Note that all these preconditioners must be computed sequentially with no parallelism for the expensive TLM and ADJ evaluations. 
\cite{Fisher_2009} present a case study of operational SC4D-VAR with spectral Lanczos preconditioner.
They note that reorthogonalization in the Lanczos method speeds up PCG convergence.
While we do not consider quasi-Newton preconditioners in this work, we refer the reader to \cite{Fisher_2009,Gratton_2011}.

\cite{bousserez2018optimal,Bousserez2020enhanced} propose methods similar to SketchSolv using RandSVD to compute the low-rank approximation, whereas our paper considers two other sketching methods and preconditioned solvers.  
\cite{Bousserez2020enhanced} also compare their results with a block-Lanczos method. 
While a block-Lanczos method allows for more parallelization of the Lanczos method, it does not allow for the same amount of parallelization as the multiple randomized methods we propose. 
However, neither paper analyzes the error in the GN descent direction. 

In~\cite{dauvzickaite2021randomised}, the authors consider the forcing formulation of the weak-constraint 4D-Var (WC-4DVAR), with a Nystr\"om as a preconditioner for PCG iterations. 
The previous ideas are extended to a state formulation of the incremental WC-4DVAR in~\cite{dauvzickaite2021time} that allows for parallelism in time. 
Randomized preconditioners for SC-4DVAR are proposed in~\cite[Chapter 5]{scott2022}. 
They compute sketches using RandSVD, although on a slightly different matrix, as they assume that the prior covariance is expensive/difficult to factorize. 
%
To the best of our knowledge, no work addresses computing adaptive randomized preconditioners for SC-4DVAR. 
Additionally, we note that our analysis is related to~\cite{frangella2021randomized}, although the analysis presented there is not in the context of DA applications. However, we also note that our analysis is more general as \cite{frangella2021randomized} focuses on the Nystr\"om approximation. Moreover, we also present and analyze algorithms for RandSVD and SingleView.

{Outside of sketching in the GN optimization algorithm, randomization has been used in the context of other optimization approaches such as Newton (\cite{pilanci2017newton,roosta2019sub,xu2020newton}), interior point methods (\cite{chowdhury2022faster}), alternating directions method of multipliers (\cite{zhao2022nysadmm}), and stochastic gradient descent (\cite{frangella2024sketchysgd}). }
\section{Background}\label{sec:background}

\subsection{Data Assimilation}
Data assimilation for state estimation combines information from three different sources: the prior best estimate, the model dynamics, and observations, to produce an improved estimate of the true state.
Formally, let us denote the unknown true state at time $t_0$ as $\B{x}^{\rm true}_0 \in \R^n$ and the prior best estimate, called the background, by $\B{x}^b_0 \in \R^n$.
The computational model---here, a set of differential equations---which evolves an initial state $\B{x}_0 \in \mathbb{R}^n$ at time $t_0$ to future state $\B{x}_i \in \mathbb{R}^n$, is denoted by  
\begin{equation}\label{eqn:model}
    \B{x}_i = \MB{M}_{t_0 \rightarrow t_i} (\B{x}_0)\,, \text{ where } \MB{M}_{t_0 \rightarrow t_i} : \R^n \to \R^n, \> 1 \leq i \leq n_t.
\end{equation}
Typically, this model is assumed inadequate, due to unknown dynamics, computational approximations, random errors, and so on. 
Observations are expensive to obtain because only a limited number of sensors can be used to collect data and are typically much smaller than the discretization of the system state, and at the same time are imperfect due to measurement errors. 
This makes the problem of estimating the initial conditions $\B{x}_0$ an ill-posed problem.
Specifically, the observations $\B{y}_i \in \R^{n_{\rm obs}}$ of the physical state are taken at times $t_i$, $1 \leq i \leq n_t$ as 
\begin{align}
    \B{y}_i = \MB{O}(\B{x}^{\rm true}_i) + \B\varepsilon^{\rm obs}_i, \text{ where } \MB{O} : \R^n \to \R^{n_{\rm obs}} \text{ and } n_{\rm obs} \ll n.
\end{align}
Here, $\MB{O}$ is the observation operator and $\B\varepsilon_i^{\rm obs}$ represents observation error.
As is the typical case, observation errors at different times are assumed to be conditionally independent~\cite[]{evensen2022data}.
From a Bayesian standpoint, we define the posterior distribution given the observational likelihood, model evolution likelihood, and the prior as:
\begin{equation}
    \pi(\B{x}_0 | \B{y}_{1}, \dots,\B{y}_{n_t} ) \propto \pi(\B{x}_0) \left( \prod_{i=1}^{n_t} \pi(\B{y}_i | \B{x}_i ) \pi(\B{x}_i | \B{x}_{i-1} ) \right) .
\end{equation}

SC-4DVAR makes certain assumptions about the model, prior, and observational likelihoods.
The first assumption is that the prior and likelihoods are normally distributed, i.e., $\B{x}_0 \sim \mc{N}(\B{x}^b_0, \B\gampr)$ and $\B\varepsilon^{\rm obs}_i \sim \mc{N}(\B{0}, \B{R}_i)$ for $1 \le i \le n_t$ and the measurements are uncorrelated in time. 
The next assumption is that of a deterministic model with no model error, implying that $\pi(\B{x}_i | \B{x}_{i-1} ) = 1$.
With these simplifying assumptions, the posterior distribution takes the form 
\begin{equation}
\begin{split}
     \pi(\B{x}_0 | \B{y}_{1}, \dots,\B{y}_{n_t}) &\propto \exp{ \left( -\frac{1}{2} \left(\B{x}_0 - \B{x}^b_0\right)^{\top} \B\gampr^{-1} \left(\B{x}_0 - \B{x}^b_0\right) \right) }\\
     &\times \exp{\left(-\frac{1}{2} \sum_{i=1}^{n_t}  \left( \mc{O}(\B{x}_i) - \B{y}_i\right)^{\top} \B{R}_i^{-1} \left( \mc{O}(\B{x}_i) - \B{y}_i\right) \right)}.
\end{split}
\end{equation}
SC-4DVAR computes the maximum a posteriori (MAP) estimate $\B{x}^{a}_0$ by solving the following optimization problem:
\begin{align}\label{eqn:4dvar}
    \B{x}^a_0 =  \argmin_{\B{x}_0\in\mb{R}^n}\,\, \mc{J}(\B{x}_0) \quad \text{subject to: }\,\,\B{x}_i = \MB{M}_{t_0 \rightarrow t_i} (\B{x}_0)\,,\quad  1 \leq i \leq n_t,
\end{align}
where the objective function is defined to be the negative logarithm of the posterior distribution (ignoring the proportionality constant, which does not affect the optimal point $\B{x}_0^a$)
\begin{align}\label{eqn:4dvar_cf}
    \mc{J}(\B{x}_0) = - \log \pi(\B{x}_0 | \B{y}_{1}, \dots,\B{y}_{n_t}) &= \frac{1}{2}  \left(\B{x}_0 - \B{x}^b_0\right)^{\top} \B\gampr^{-1} \left(\B{x}_0 - \B{x}^b_0\right) \nonumber \\
     &+ \frac{1}{2} \sum_{i=1}^{n_t} \left( \mc{O}(\B{x}_i) - \B{y}_i\right)^{\top} \B{R}_i^{-1} \left( \mc{O}(\B{x}_i) - \B{y}_i\right).
\end{align}
The first term in \cref{eqn:4dvar_cf} quantifies the departure of the solution from the background state $\B{x}^b_0$ and the second term measures the mismatch between the forecast trajectory and observations in the assimilation window. 
The minimizer of \cref{eqn:4dvar} is typically computed iteratively using gradient-based numerical optimization methods. 
First-order adjoint models provide the gradient of the cost function \cref{eqn:4dvar_cf}, while second-order adjoint models provide the Hessian-vector product (e.g., for Newton-type methods). 
The methodology for building and using various adjoint models for optimization and sensitivity analysis is discussed in \cite{sandu2005adjoint,Sandu_PA2006a,cioaca2012second}.
This paper considers the approach in \cite{Sandu_PA2006a} for constructing discrete Runge-Kutta adjoints.
\subsection{Newton-based approaches}
In this work, we focus on the inexact or truncated GN method~\cite[]{nocedal1999numerical,chen2011gnhess} for two main reasons \begin{enumerate*}[label={(\roman*)}]
    \item it is cheaper than the full Newton method as it does not require the computation of second-order derivatives of the model,
    \item it is better suited for small residual problems and is robust to the initial guess for the optimization problem in \eqref{eqn:4dvar}. 
\end{enumerate*}
The gradient of \cref{eqn:4dvar_cf} can be written as 
\begin{equation}\label{eqn:4dvar_gradient}
    \nabla_{\B{x}_0} \mc{J}(\B{x}_0) = \displaystyle \B\gampr^{-1} \left(\B{x}_0 - \B{x}^b_0\right) + \sum_{i=1}^{n_t} \B{M}\t_{i}(\B{x}_0) \B{O}\t_i(\B{x}_i)  \B{R}_i^{-1} \left( \MB{O}(\B{x}_i) - \B{y}_i\right)\,,
\end{equation}
where $\B{M}\t_{i}(\B{x}_0) = \left. \frac{\partial \MB{M}\t_{t_0 \rightarrow t_i}}{\partial \B{x}}  \right|_{\B{x}_0}$ and $\B{O}\t_i(\B{x}_i) = \left.\frac{\partial\mc{O}\t}{\partial \B{x}}\right|_{\B{x}_i}$ are the adjoints (ADJs) of the model and observation operators, respectively. 
The GN Hessian of \cref{eqn:4dvar_cf} with respect to $\B{x}_0$ is 
\begin{align}\label{eqn:4dvar_hessian}
    \MB{H}(\B{x}_0) = & \> \B\gampr^{-1}  + \sum_{i=1}^{n_t} \B{M}\t_{i}(\B{x}_0) \B{O}\t_i(\B{x}_i) \B{R}_i^{-1}  \B{O}_i(\B{x}_i) \B{M}_i(\B{x}_0) \,.
\end{align}
where $\B{M}_{i}(\B{x}_0) = \left. \frac{\partial \MB{M}_{t_0 \rightarrow t_i}}{\partial \B{x}}  \right|_{\B{x}_0}$ and $\B{O}_i(\B{x}_i) = \left.\frac{\partial\mc{O}}{\partial \B{x}}\right|_{\B{x}_i}$ are the tangent linear models (TLMs) of the model and observation operators respectively. 
Note that the GN Hessian can be factorized as 
\begin{equation}\label{eqn:Rhalf}
\begin{split}
    \MB{H}(\B{x}_0) &= \B\gampr^{-1/2} \left( \B{I} + \B{A}\t(\B{x}_0)\B{A}(\B{x}_0) \right) \B\gampr^{-1/2} = \B\gampr^{-1/2} \left( \B{I} + \B{H}(\B{x}_0) \right) \B\gampr^{-1/2}, \\
    \B{A}(\B{x}_0) &= \bmat{
    \B{R}_1^{-1/2}\B{O}_1(\B{x}_1)\B{M}_1(\B{x}_0) \B\gampr^{1/2} \\ 
    \B{R}_2^{-1/2}\B{O}_2(\B{x}_2)\B{M}_2(\B{x}_0) \B\gampr^{1/2}\\ 
    \vdots \\
    \B{R}_{n_t}^{-1/2} \B{O}_N(\B{x}_N)\B{M}_{n_t}(\B{x}_0) \B\gampr^{1/2}} \in \R^{m\times n}, \quad m = n_{\rm obs} n_t,
\end{split} 
\end{equation}
where $\B{H}(\B{x}_0) = \B{A}\t(\B{x}_0)\B{A}(\B{x}_0)$, is the data misfit part of the GN Hessian. 
As stated in the introduction, the $\B{H}(\B{x}_0)$ term exhibits a rapid decay in its eigenvalues; see e.g.,~\cite{flath2011fast, Tanbui_2013,spantini2015optimal,cui2014likelihood}).
This is also discussed in \Cref{ssec:analysis_sum}.

An evaluation of the gradient involves a sequential computation of the full model forward in time and the adjoint model backward in time, while an evaluation of the GN Hessian requires a solution of the full model and its tangent linear model forward in time, along with the adjoint model backward in time. 
The initial state at the $k$-th GN iteration is denoted by $\B{x}_0^{(k)}$, and its corresponding descent direction by $\delta\B{x}^{(k)}$.
This descent direction requires the solution of the GN linear system 
\begin{equation}\label{eqn:gniter} 
    \MB{H}^{(k)} \delta\B{x}^{(k)} = - \B{g}^{(k)},
\end{equation}
where $\MB{H}^{(k)} = \MB{H}(\B{x}_0^{(k)})$, and $\B{g}^{(k)} = \nabla_{\B{x}_0} \mc{J}(\B{x}_0^{(k)})$. 
Using the factorization from \cref{eqn:Rhalf} in \cref{eqn:gniter} gives a prior-preconditioned (or first-level preconditioned ~\cite{Tshimanga_2008,Fisher_2009}) linear system for 
\begin{equation}\label{eqn:gniterf1} 
     \left( \B{I} + \B{H}^{(k)} \right) \B\gampr^{-1/2} \delta\B{x}^{(k)} = - \B\gampr^{1/2} \B{g}^{(k)},
\end{equation}
where $\B{H}^{(k)} = \B{H}(\B{x}_0^{(k)})$.

As the following discussion is for one GN iteration, we omit the superscript $(k)$ in the notation.
Assume the existence of an approximate low-rank eigenvalue decomposition of $\B{H} \approx \Bh{H} = \Bh{V} \Bh{\Lambda} \Bh{V}\t$. 
Applying the Woodbury identity gives    
\begin{equation}\label{eqn:woodbury}
    (\B{I} + \Bh{H} )^{-1} = (\B{I} + \Bh{V} \Bh{\Lambda} \Bh{V}\t )^{-1} = \left( \B{I} - \Bh{V}({\Bh{\Lambda}}^{-1} + \B{I})^{-1}\Bh{V}\t \right),
\end{equation}
which only requires the inverse of a diagonal matrix $({\Bh{\Lambda}}^{-1} + \B{I})^{-1}$. 
The approximation $(\B{I} + \B{H} )^{-1} \approx \left( \B{I} - \Bh{V}({\Bh{\Lambda}}^{-1} + \B{I})^{-1}\Bh{V}\t \right)$, can be used in two ways. 
\subsubsection{SketchSolv}
\label{subsubsec:sksolv}
Here, the GN step direction is obtained by directly applying the $(\B{I} + \Bh{H})^{-1}$ from \cref{eqn:woodbury} to \cref{eqn:gniterf1} at the $k$-th iterate as 
\begin{equation}\label{eqn:dxsol}
    \delta\Bh{x}  = - \B\gampr^{1/2} \left( \B{I} - \Bh{V}({\Bh{\Lambda}}^{-1} + \B{I})^{-1}\Bh{V}\t \right) \B\gampr^{1/2}\B{g}.
\end{equation}
Like inexact GN, the optimization proceeds by using an inexact GN step that can be computed more efficiently than the exact GN step. 
Combining the convergence analysis of inexact GN with the guarantees using the SketchSolv approach is possible, but we did not pursue this analysis (ideas explored in \cite{Gratton2007}).
\subsubsection{SketchPrec}
\label{subsubsec:skprec}
In this approach, the system in \cref{eqn:gniterf1} is solved using the {conjugate gradient} method as 
\begin{equation}\label{eqn:pcgsystem}
    \left( \B{I} + \B{H} \right) \delta \Bt{x} = - \B\gampr^{1/2} \B{g}, \text{ and } \delta \B{x} = \B\gampr^{1/2} \delta \Bt{x}. 
\end{equation}
The conjugate gradient method for \cref{eqn:pcgsystem} is preconditioned by $(\B{I} + \Bh{H})^{-1}$ from \cref{eqn:woodbury}. {This approximation is known as a limited memory preconditioner and has been used to solve the {SC-4DVar problem}~\cite[]{Tshimanga_2008, gratton2011reduced}.}

\subsubsection{Low-rank approximations}
{
We now address how to efficiently obtain a low-rank approximation $\B{H} \approx \Bh{H} = \Bh{V} \Bh{\Lambda} \Bh{V}\t$.
Past work has focused on the Rayleigh-Ritz approaches (including the Lanczos method)~\cite[]{Tshimanga_2008,Fisher_2009}, both of which are inherently serial and allow for no parallelization.
Another method is a variant of the Rayleigh-Ritz where the eigenvalue decomposition (via the Lanczos algorithm) is replaced by a randomized eigenvalue decomposition~\cite[]{scott2022,dauvzickaite2021randomised}.
%
%
We present randomized algorithms that allow for parallelism in the TLM and ADJ computations and adaptivity in the sketch sizes.}
\section{Methodology}
\label{sec:methodology}
In this section, we discuss randomized methods for creating $\Bh{H}$ in \Cref{ssec:sketching}, and its application to \Cref{subsubsec:skprec,subsubsec:sksolv} in \Cref{ssec:applopt}.

\subsection{Sketching the GN Hessian}\label{ssec:sketching}
{
As in the previous section, we omit the superscript $(k)$ from the notation as this discussion consists of a single GN iteration. 
Let the current iterate be denoted as $\B{x} \in \R^n$ and define the preconditioned forward operator $\B{A}$, its adjoint $\B{A}\t$ have the same meaning as before. 
The randomized methods rely on the application of $\B{A}$ and $\B{A}\t$ to a vector, without access to the full matrix itself. 
Throughout this section, we will assume that the random matrices are standard Gaussian, meaning that the entries are independent and identically distributed random variables drawn from the Gaussian distribution with zero mean and unit variance. 
%
%
All three methods presented here---RandSVD, Nystr\"om, and SingleView---produce a low-rank eigenvalue decomposition of the form $\widehat{\B{H}} = \Bh{V}\Bh{\Lambda}\Bh{V}\t$, where $\Bh{V} \in \R^{n \times \ell}$ has orthonormal columns, and  $\Bh{\Lambda} \in \R^{\ell \times \ell}$ has nonnegative diagonal entries.
The choice of $\ell$ can be \textit{fixed} or \textit{adaptive} based on the specific application.
As $\B{H}$ will have a spectrum with many small eigenvalues, $\ell \ll m$ can be reasonably small (further discussed in \Cref{ssec:analysis_sum}). 
}

\subsubsection{Fixed-rank sketching}
{In this section, we discuss methods to sketch $\Bh{H}$, where the sketch parameter $\ell$ is fixed.}
\paragraph{Randomized SVD:} 
This method computes a low-rank approximation of the preconditioned forward operator $\B{A}$ using the randomized SVD approach in~\cite{halko2011finding}. 
First, we draw a random matrix $\B\Omega \in \mb{R}^{n\times \ell}$. 
Next, we compute the sketch $\B{Y} = \B{A\Omega}$; the key insight is that if $\B{A}$ is low-rank, then the columns of $\B{Y}$ approximate the range of $\B{A}$. 
To obtain a low-rank approximation, we consider the thin-QR factorization $\B{Y} = \B{QR}$, then obtain the low-rank approximation $\B{A}\approx \B{QQ}^\top \B{A}$; similarly, we also have the low-rank approximation 
\begin{equation} 
    \B{H} = \B{A}^\top \B{A} \approx \Bh{H} \equiv \B{A}^\top \B{QQ}^\top \B{A} = \B{W}\t \B{W}
\end{equation}
where $\B{W}\t = \B{A}\t\B{Q}$.
This approach requires $\ell$ TLM applications followed by $\ell$ ADJ applications---both of which can be computed in parallel. 
The computation of the QR factorization is $\mc{O}(mr^2)$ floating point operations (flops), and there is an additional postprocessing cost $\mc{O}(nr^2 + r^3)$ flops to convert to SVD format. {Here, $r$ is the target rank and $\ell = r + p$, with $p$ being the oversampling parameter}. 
The details of this approach are given in \Cref{alg:randsvd}.
\begin{algorithm}[!ht]
    \begin{algorithmic}[1]
        \REQUIRE Operator: $\B{A}\in\R^{m\times n}$, sketch parameter: $\ell$.
        \STATE Draw Gaussian random matrix $\B\Omega \in \mb{R}^{n\times \ell}$.
        \STATE Compute forward pass $\B{Y} = \B{A\Omega}$. \COMMENT{Compute, if possible, in parallel}
        \STATE Compute thin-QR factorization $\B{Y} = \B{QR}$. 
        \STATE Compute $\B{W}\t = \B{A}\t\B{Q}$. \COMMENT{Compute, if possible, in parallel}
        \STATE Compute thin-SVD $\B{W} = \Bh{U}\Bh{\Sigma}\Bh{V}\t $, $\Bh{\Lambda} = \Bh\Sigma^2$.
         \RETURN Matrices $\Bh{V}, \Bh{\Lambda}$ that defines the approximation $\Bh{H} = \Bh{V}\Bh{\Lambda}\Bh{V}\t$.
        \caption{RandSVD GN Hessian Approximation}
    \label{alg:randsvd}
    \end{algorithmic}
\end{algorithm}

\begin{remark}\label{rem:thinsvdavoid}
    We may avoid thin-SVDs of $\B{W}$ by using $\Bh{H} = \B{W}\t\B{W}$, but this requires a solution of  linear system of size $\ell \times \ell$ as the Woodbury identity gives 
    \begin{equation}\label{eqn:randsvdsv}
        (\B{I} + \Bh{H})^{-1} = \B{I} - \B{W}\t (\B{I}_{\ell} + \B{W} \B{W}\t)^{-1} \B{W}.
    \end{equation}
    $(\B{I}_{\ell} + \B{W} \B{W}\t)^{-1}$ requires a Cholesky decomposition ($\mathcal{O}(r^3)$) which should be cheaper than a thin-SVD ($\mathcal{O}(nr^2)$).
\end{remark}

\paragraph{Nystr\"om:}
In this approach, we first form the sketch of the prior preconditioned GN Hessian, $\B{H}$, by drawing a standard Gaussian random matrix $\B\Omega \in \R^{n\times \ell}$ and computing $\B{Y} = \B{H\Omega}$. 
Notice that it requires $\ell$ TLM followed by $\ell$ adjoint solves, which can be done in parallel. 
The Nystr\"om approach, see e.g., eq. (14.1) in ~\cite{martinsson2020randomized}, forms the following low-rank approximation to $\B{H}$ as 
\begin{equation}
    \B{H} \approx \B{H}_{\rm Nys} \equiv \B{Y} (\B\Omega\t\B{Y})^\dagger \B{Y}\t = \B{H\Omega}(\B\Omega\t\B{H\Omega})^\dagger (\B{H\Omega})\t,
\end{equation}
where $^\dagger$ represents the Moore-Penrose inverse. 
However, the expression implemented \textit{as is} has numerical instability issues, so we follow the approach in Algorithm 3 from \cite{tropp2017fixed} -- which is reproduced in \Cref{alg:nystrom}. 
The computational cost of the Nystr\"om approach is similar to RandSVD; namely, $\ell$ TLM followed by $\ell$ ADJ solves followed by $\mc{O}(r^2n + r^3)$ flops for post-processing. 
 
\begin{algorithm}[!ht]
    \begin{algorithmic}[1]
    \REQUIRE Operator: $\B{H} = \B{A}\t\B{A} \in \R^{n\times n}$, sketch parameter: $\ell$. 
        \STATE Draw Gaussian random matrix $\B\Omega \in \mb{R}^{n\times \ell}$.
        \STATE Compute $\B{Y} = \B{H\Omega}$. \COMMENT{Compute, if possible, in parallel}
        \STATE Compute $\nu = \sqrt{n}\epsilon_{\rm mach}\|\B{Y}\|_2$ where $\epsilon_{\rm mach}$ is machine epsilon.
        \STATE $\B{Y} \leftarrow\B{Y} + \nu \B\Omega$.
        \STATE Compute (lower) Cholesky triangular factor $\B{L}$ of $(\B{C}+\B{C}\t)/2$, where $\B{C} = \B\Omega\t\B{Y}$.
        \STATE {Compute $\B{L}^{-1} \B{Y}\t = \B{W}$.
        \STATE Compute the thin-SVD (truncated to rank r) of $\B{W}= \Bh{U}\Bh\Sigma\Bh{V}\t$.}
        \RETURN Matrices $\Bh{V}, \Bh{\Lambda} = \max\{0,\Bh\Sigma^2 {-\nu\B{I}}\}$ that defines the approximation $\Bh{H} = \Bh{V}\Bh{\Lambda}\Bh{V}\t$.
    \end{algorithmic}
    \caption{Nystr\"om GN Hessian Approximation}
    \label{alg:nystrom}
\end{algorithm}

\paragraph{Single View:}

This approach is based on the Single View randomized algorithm in~\cite{tropp2017practical}. In this approach, we draw two standard Gaussian random matrices $\B\Omega \in \mb{R}^{n\times \ell_1}$, $\B\Psi \in \mb{R}^{m\times \ell_2}$; the sketch parameters $\ell_1,\ell_2$ will be discussed shortly. 
As in RandSVD, we compute the sketch $\B{Y} = \B{A\Omega}$ and its thin-QR factorization $\B{Y} = \B{Q}_Y\B{R}_Y$. 
The SingleView method maintains a separate sketch for the row space $\B{Z} = \B{A}\t\B\Psi$ and constructs the low-rank approximation 
\begin{equation}
    \label{eqn:singleviewlr}
    \B{A} \approx \Bh{A}_{SV} \equiv \B{Q}_Y (\B\Psi\t\B{Q}_Y)^\dagger \B{Z}\t.
\end{equation}
A main advantage of this approach is that both the sketches $\B{Y}$ and $\B{Z}$ can be computed independently and in parallel. 
To understand this approach, it is worth noting that the SingleView approximation can be expressed as $\Bh{A}_{SV} \equiv \B{P}_{SV}\B{A}$, where $\B{P}_{SV} = \B{Q}_Y (\B\Psi\t\B{Q}_Y)^\dagger \B{\Psi}\t$ is an oblique projector with $\range(\B{P}_{SV})\subset \range(\B{Q}_Y)$, {where $\range(\cdot)$ denotes the range of a matrix}. 
In contrast, the RandSVD approach uses an orthogonal projector onto $\range(\B{Q}_Y)$. 
We now discuss an efficient implementation of the low-rank approximation. 

Compute $\B{M} = \B\Psi\t\B{Q}_Y$ and its thin-QR factorization $\B{M}= \B{Q}_M\B{R}_M$; its pseudoinverse can be computed as $\B{M}^\dagger = \B{R}_M^\dagger \B{Q}_M\t$. 
Now, we compute $\B{W} = \B{R}_M^\dagger (\B{ZQ}_M)\t$ and its thin-SVD $\B{W} = \B{U}_W\B\Sigma\Bh{V}\t$. 
In summary, so far we have $\B{A}_{SV} = (\B{Q}_Y \B{U}_W) \B\Sigma \Bh{V}\t$. 
The approximation of the Gram product is now straightforward
\begin{equation*}
    \B{H} \approx \Bh{H}_{SV} \equiv \B{A}_{SV}\t\B{A}_{SV} = \Bh{V}\Bh\Lambda\Bh{V}\t,   
\end{equation*}
where $\Bh\Lambda = \B\Sigma^2$. The details of this approach are given in Algorithm~\ref{alg:singleview}.

Because the oblique projector introduces some numerical issues, a standard approach is to use additional oversampling. 
Based on eq.\ (4.6) with $\alpha = 1$ from \cite{tropp2017practical}, one recommendation is to use $\ell_1 = 2r + 1$ and $\ell_2 =  2\ell_1 + 1$, where $r$ is the target rank. 
In our numerical experiments, we take $\ell_1$ to be equal to $\ell$ as in the previously discussed RandSVD and Nystr\"om methods to ensure all the methods have comparable cost to sketch, specifically, the same number of TLM vector products.
Note that this approach requires more ADJ solves (and TLM solves if $\ell_1 > \ell$) than the previously discussed methods, but the key point here is that these ADJ and TLM solves can each be performed in parallel, thus exposing an additional layer of parallelism.
\begin{algorithm}[!ht]
    \begin{algorithmic}[1]
        \REQUIRE Operator: $\B{A}\in\R^{m\times n}$, sketch parameters: $\ell_1, \ell_2$. 
        \STATE Draw Gaussian random matrices $\B\Omega \in \mb{R}^{n\times \ell_1}$, $\B\Psi \in \mb{R}^{m\times \ell_2}$.
        \STATE Compute sketches $\B{Y} = \B{A\Omega}$ and $\B{Z} = \B{A}\t \B\Psi$ .\COMMENT{Compute, if possible, in parallel}
        \STATE Compute thin-QR factorization of $\B{Y} = \B{Q}_Y\B{R}_Y$.
        \STATE Compute thin-QR factorization of $\B{M} = \B{\Psi}\t\B{Q}_Y = \B{Q}_M\B{R}_M$.
        \STATE Compute $\B{W} = \B{R}_M^\dagger (\B{ZQ}_M)\t $ and its thin-SVD $\B{W}=\Bh{U} \Bh\Sigma \Bh{V}\t$.
        \RETURN Matrices $\Bh{V}, \Bh{\Lambda} = \B\Sigma^2$ that defines the approximation $\Bh{H}_{SV} = \Bh{V}\Bh{\Lambda}\Bh{V}\t$.
    \end{algorithmic}
    \caption{SingleView GN Hessian Approximation}
    \label{alg:singleview}
\end{algorithm}

\subsubsection{Adaptive sketching}
\label{subsubsec:adap}
In many applications, the spectrum of $\B{H}$, can vary across the iterates $(k)$, and the choice of a fixed $\ell$ may be inappropriate. 
To this end, we design a methodology to adaptively choose the sketch size $\ell$ based on the condition number of the GN Hessian in \cref{eqn:gniter}. 
We use the following quantity $\kappa_{\rm sk}$ as an estimate for the accuracy of $\Bh{H}$
\begin{equation}\label{eqn:condtext}
\begin{split}
    \kappa_{\rm sk} &= \lVert (\B{I} + \B{H}) (\B{I} + \Bh{H})^{-1} \B{v} \lVert_2 
\end{split}
\end{equation}
where $\B{v} = \B{w}/\|\B{w}\|_2$  and $\B{w} \sim \mc{N}(\B{0},\B{I})$,  is a normal random test vector. 
In \Cref{ssec:cdest}, we show that $\kappa_{\rm sk}$ can be viewed as an estimator for the 2-norm condition number of $(\B{I} + \B{H}) (\B{I} + \Bh{H})^{-1}$. 
To compute the inverse of $\B{I}+\Bh{H}$, we use the Woodbury formula as in~\eqref{eqn:woodbury}.

The idea behind the adaptive sketching approach is 
\begin{enumerate*}[label={(\roman*)}]
    \item to construct  a sketch $\Bh{H}$ with an initial choice of $\ell$;
    \item while $\kappa_{\rm sk} > \epsilon_{\rm sk}$, recompute a new sketch with a larger $\ell$; 
    Here $\epsilon_{\rm sk} > 1$ is a user-specified parameter that is used to estimate the condition number. 
\end{enumerate*}
All three randomized methods can be modified to adaptively estimate the sketch size.
The details for adapting RandSVD are provided in~\Cref{alg:arandsvd}; the implementation for Nystr\"om is similar and given in \cite{frangella2021randomized}.
The adaptive algorithm for SingleView is provided in~\Cref{alg:asingleview}.

\begin{algorithm}[!ht]
    \begin{algorithmic}[1]
        \REQUIRE Operator: $\B{A}\in\R^{m\times n}$, sketch parameters: $\ell$, $\hat{\ell}$, {tolerance: $\epsilon_{\rm sk} > 0$.}
        \STATE Draw Gaussian random matrix $\B\Omega \in \mb{R}^{n\times \ell}$.
        \STATE Compute forward pass $\B{Y} = \B{A\Omega}$. \COMMENT{Compute, if possible, in parallel}
        \STATE Compute thin-QR factorization $\B{Y} = \Bh{Q}\B{R}$. 
        \STATE Compute $\B{W}\t = \B{A}\t\Bh{Q}$. \COMMENT{Compute, if possible, in parallel}
        \STATE Initialize the empty matrix as $\B{Q} = \begin{bmatrix}
            &
        \end{bmatrix}$.
        \STATE {Calculate $\kappa_{\rm sk}$ using \cref{eqn:condtext,eqn:randsvdsv}.} 
        \WHILE {$\kappa_{\rm sk} > \epsilon_{\rm sk}$} 
            \STATE Append to $\B{Q} \leftarrow \bmat{\B{Q} & \Bh{Q}}$.
            \STATE Draw Gaussian random matrix $\Bh{\Omega} \in \mb{R}^{n\times \hat{\ell}}$.
            \STATE Compute forward pass $\B{Y} = \B{A}\Bh{\Omega}$. \COMMENT{Compute, if possible, in parallel}
            \STATE Orthogonalize $\B{Y} = \B{Y} - \B{Q}(\B{Q}^\top \B{Y})$. \COMMENT{Additional Gram-Schmidt orthogonalization steps, if necessary}
            \STATE Compute thin-QR factorization $\B{Y} = \Bh{Q}\B{R}$.
            \STATE Compute $\Bh{W}\t = \B{A}\t \Bh{Q}$. \COMMENT{Compute, if possible, in parallel}
            \STATE Update $\ell \leftarrow \ell + \hat{\ell}$, $\B{W} \leftarrow \bmat{\B{W}\t & \Bh{W}\t}\t$.
            \STATE {Calculate $\kappa_{\rm sk}$ using \cref{eqn:condtext,eqn:randsvdsv}.}
        \ENDWHILE
        \STATE Compute thin-SVD $\B{W} = \Bh{U}\Bh{\Sigma}\Bh{V}\t $, $\Bh{\Lambda} = \Bh\Sigma^2$.
         \RETURN Matrices $\Bh{V}, \Bh{\Lambda}$ that defines the approximation $\Bh{H} = \Bh{V}\Bh{\Lambda}\Bh{V}\t$.
        \caption{Adaptive RandSVD GN Hessian Approximation}
    \label{alg:arandsvd}
    \end{algorithmic}
\end{algorithm}

\begin{algorithm}[!ht]
    \begin{algorithmic}[1]
        \REQUIRE Operator: $\B{A}\in\R^{m\times n}$, sketch parameters: $\ell_1, \ell_2, \hat{\ell}_1, \hat{\ell}_2${, tolerance: $\epsilon_{\rm sk} > 0$} 
        \STATE Draw Gaussian random matrices $\B\Omega \in \mb{R}^{n\times \ell_1}$, $\B\Psi \in \mb{R}^{m\times \ell_2}$.
        \STATE Compute sketches $\B{Y} = \B{A\Omega}$ and $\B{Z} = \B{A}\t \B\Psi$.\COMMENT{Compute, if possible, in parallel}
        \STATE Compute thin-QR factorization of $\B{Y} = \B{Q}_Y\B{R}_Y$.
        \STATE Compute thin-QR factorization of $\B{M} = \B{\Psi}\t\Bh{Q}_Y = \B{Q}_M\B{R}_M$.
        \STATE Compute $\B{W} = \B{R}_M^\dagger (\B{ZQ}_M)\t $.
        \STATE {Calculate $\kappa_{\rm sk}$ using \cref{eqn:condtext,eqn:randsvdsv}.} 
        \WHILE {$\kappa_{\rm sk} > \epsilon_{\rm sk}$} 
            \STATE Draw Gaussian random matrices $\Bh{\Omega} \in \mb{R}^{n\times \hat{\ell}_1}$, $\Bh{\Psi} \in \mb{R}^{m\times \hat{\ell}_2}$.
            \STATE Compute forward pass $\B{Y} = \B{A}\Bh{\Omega}$. \COMMENT{Compute, if possible, in parallel}
            \STATE Orthogonalize $\B{Y} = \B{Y} - \B{Q}_Y(\B{Q}^\top_Y \B{Y})$. \COMMENT{Additional Gram-Schmidt orthogonalization steps.}
            \STATE Compute thin-QR factorization $\B{Y} = \Bh{Q}_Y\Bh{R}_Y$.
            \STATE {Compute thin-QR factorization of $\begin{bmatrix}
                \B\Psi\t \\ \Bh\Psi\t
            \end{bmatrix} \begin{bmatrix}
                \B{Q}_Y & \Bh{Q}_Y
            \end{bmatrix} = \B{Q}_M \B{R}_M$ directly or via \cref{alg:asingleviewqrupdate} since the QR-decomposition of $\B\Psi\t \B{Q}_Y$ is known.}
            \STATE Compute $\Bh{W}\t = \B{A}\t \Bh{Q}$. \COMMENT{Compute, if possible, in parallel}
            \STATE Compute $\B{W} = \B{R}_M^\dagger (\B{ZQ}_M)\t $. 
            \STATE Update some variables. $\ell \leftarrow \ell + \hat{\ell}$, $\B{W} \leftarrow \bmat{\B{W}\t & \Bh{W}\t}\t$.
            \STATE {Calculate $\kappa_{\rm sk}$ using \cref{eqn:condtext,eqn:randsvdsv}.} 
        \ENDWHILE
        \STATE Compute thin-SVD $\B{W} = \Bh{U}\Bh{\Sigma}\Bh{V}\t $, $\Bh{\Lambda} = \Bh\Sigma^2$.
         \RETURN Matrices $\Bh{V}, \Bh{\Lambda}$ that defines the approximation $\Bh{H} = \Bh{V}\Bh{\Lambda}\Bh{V}\t$.
        \caption{Adaptive SingleView GN Hessian Approximation}
    \label{alg:asingleview}
    \end{algorithmic}
\end{algorithm}

\begin{algorithm}[!ht]
    \begin{algorithmic}[1]
        \REQUIRE {Matrices: $\begin{bmatrix}
            \B{A}_1 \\ \B{A}_2
        \end{bmatrix}$, $\begin{bmatrix}
            \B{B}_1 & \B{B}_2
        \end{bmatrix}$ where $\B{B}_1, \B{B}_2$ are orthogonal, and $\B{Q}_1 \B{R}_1 = \B{A}_1 \B{B}_1$.} 
        {
        \STATE Compute thin-QR of $\B{A}_2 \B{B}_1 = \B{Q}_2 \B{R}_2$.
        \STATE Compute thin-QR of $\begin{bmatrix}
            \B{R}_1 \\ \B{R}_2
        \end{bmatrix} = \B{Q}_3 \B{R}_3$.
        \STATE Reset $\B{Q}_3 = \begin{bmatrix}
            \B{Q}_1 & \B{0} \\ \B{0} & \B{Q}_2
        \end{bmatrix} \B{Q}_3$, giving the thin-QR decomposition $\begin{bmatrix}
            \B{A}_1 \\ \B{A}_2
        \end{bmatrix} \B{B}_1 = \B{Q}_3 \B{R}_3$. 
        \STATE Compute $\B{P}_1 = \begin{bmatrix}
            \B{A}_1 \\ \B{A}_2
        \end{bmatrix} \B{B}_2$.
        \STATE Orthogonalize $\B{P}_2 = \B{P}_1 - \B{Q}_3 ( \B{Q}\t_3 \B{P}_1)$.
        \STATE Compute thin-QR of $\B{P}_2 = \B{Q}_4 \B{R}_4$.
        \STATE Compute matrices $\B{Q}_5 = \begin{bmatrix}
            \B{Q}_3 & \B{Q}_4
        \end{bmatrix}$ and $\B{R}_5 = \begin{bmatrix}
            \B{R}_3 & \B{Q}\t_3\B{P}_1\\ \B{0} & \B{R}_4
        \end{bmatrix}$.}
        \RETURN {Matrices $\B{Q}_5, \B{R}_5$ where $ \B{Q}_5 \B{R}_5 = \begin{bmatrix}
            \B{A}_1 \\ \B{A}_2
        \end{bmatrix} \begin{bmatrix}
            \B{B}_1 & \B{B}_2
        \end{bmatrix}$.}
    \caption{{QR Update for adaptive SingleView}}
    \label{alg:asingleviewqrupdate}
    \end{algorithmic}
\end{algorithm}

\subsubsection{Comparison of computational costs} 
We compare the proposed approaches with the traditional Lanczos method in terms of the computational costs, storage, and scope for parallelism. 
We assume that each TLM and ADJ run costs the same.
In Table~\ref{tab:compcosts}, we summarize the computational costs, and in reporting the costs we assume that $\ell, \ell_1,\ell_2 = \mc{O}(r)$.
As can be seen, in a serial implementation, the computational costs of RandSVD and Nystr\"om are very similar. 
Furthermore, SingleView is the most expensive of the three algorithms, although it is more parallelizable.

\begin{table}[!ht]
    \centering
    \begin{tabular}{c|c|c|c|c|c}
      Method & TLM & ADJ &  Additional (flops) &  Storage &  PPE ($\%$)\\ \hline
      Lanczos & $\ell$ & $\ell$ & $\mc{O}(r^3 + r^2n)$ & $n\ell + 3\ell$ & 0\\
      RandSVD & $\ell$ & $\ell$ & $\mc{O}(r^3 + r^2 (m+n) )$ & $2n\ell$ &  $50$ \\
      Nystr\"om & $\ell$ & $\ell$ & $\mc{O}(r^3 + r^2n)$ & $2n\ell$ &  $50$ \\
      SingleView & $\ell_1 $ & $\ell_2$ & $\mc{O}(r^3  + r^2 (m+n) )$  &  $2n\ell_1$ & $100$ \\
    \end{tabular}
    \caption{Comparison of computational costs for creating a sketch $\Bh{H}$. All the methods apply the square root of the prior ($\B{\gampr}^{1/2}$) $2 \ell$ times. PPE refers to peak parallel efficiency. The cost for the Lanczos method assumed full reorthogonalization.}
    \label{tab:compcosts}
\end{table}

Assume that we have a sufficiently large number of processors available. 
We define \textbf{peak parallel efficiency} (or PPE) as the percentage of the maximum number of PDE solves that can be run in parallel at any given time to the total number of PDE solves that need to be run. 
Then, for RandSVD, we can perform $\ell$ TLM solves in parallel (to compute $\B{Y} = \B{A\Omega}$) followed $\ell$ ADJ solves in parallel (to compute $\B{Z} = \B{A}\t\B{Q}$). 
This gives this method the PPE of $50\%$. 
The Nystr\"om approach similarly has a PPE of $50\%$. 
However, in the SingleView approach, both the TLM and ADJ solves can be done simultaneously in parallel, so it has a PPE of $100\%$. 
The cost analysis does not consider the cost of adaptivity because this cost is negligible.

There are other forms of parallelism in solving PDEs~\cite[]{Gander_2015_review,rao2016parallel} such as parallelism in the spatial and temporal dimensions that are not explored here but may further increase the gains from parallelism from our approach.

\subsection{Application to optimization}
\label{ssec:applopt}

In the most basic version, the sketch $\Bh{H}^{(k)}$ is computed at every GN iteration for both SketchSolv(as in~\Cref{subsubsec:sksolv}) and SketchPrec (as in~\Cref{subsubsec:skprec}).
While it looks as expensive as the traditional Lanczos method, the advantage of randomized sketching is the parallel evaluation of the expensive TLM and ADJ models. 
Additionally, randomization accelerates the overall convergence of GN in SketchPrec~\Cref{subsubsec:skprec} paradigm~\cite[Section 10.5]{martinsson2020randomized}.

One may also reuse the sketch $\Bh{H}^{(k)}$ across multiple GN iterations.
We do not consider reusing $\Bh{H}^{(k)}$ in the SketchSolv approach as the inexactness of the step direction may cause convergence issues.
However, reusing a previous sketch in the SketchPrec approach would only slow down the PCG convergence, without affecting the accuracy. 
But as the experiments will show, this cost is minimal compared to the cost of resketching $\Bh{H}^{(k)}$.

One needs a methodology to decide if a sketch $\B{H}^{(k - i)}$ must be reused or recomputed. 
Similar to~\cref{eqn:condtext}, we look at the conditioning of the linear system as in \cref{eqn:pcgsystem} using a sketch from a previous $(k - i)$-th GN iteration as 
\begin{equation}\label{eqn:condtext2}
    \kappa_{\rm re} = \lVert (\B{I} + \B{H}^{(k)}) (\B{I} + \Bh{H}^{(k - i)})^{-1} \B{v} \lVert_2
\end{equation}
where $0 \leq i < k$ where, as before, $\B{v} = \B{w}/\|\B{w}\|_2$ and $\B{w} \sim \mc{N}(\B{0},\B{I})$. 
The heuristic we propose is to reuse the sketch if  $\kappa_{\rm re} < \epsilon_{\rm re}$, where $\epsilon_{\rm re}$ is a user-specified tolerance.

One may consider sketch reuse with both fixed or adaptive sketch sizes.
In this work, SketchPrec considers either a fixed sketch size with no reuse or an adaptive sketch size with reuse (called SketchPrecA).
The two parameters $\kappa_{\rm sk}$ and $\kappa_{\rm re}$ are tuned as follows. 
Firstly, one must choose $1 \leq \epsilon_{\rm sk} < \epsilon_{\rm re}$, as the decision of sketch size must be stricter than the decision of reuse. 
This is because a ``more accurate'' (based on $\epsilon_{\rm sk}$) preconditioner from a previous GN iteration could still be somewhat adequate as a preconditioner (based on $\epsilon_{\rm re}$) in the current GN iteration, while a less accurate precondition may result in expensive resketching. 
Due to the quantities $\epsilon_{\rm sk}$ and $\epsilon_{\rm re}$ being representative of condition number, we typically choose $\epsilon_{\rm sk} \to 1^+$ to tend to 1 from the right, and $\epsilon_{\rm re} \approx 5 - 100$.
Ultimately, the choice of $\epsilon_{\rm sk}$ and $\epsilon_{\rm re}$ is problem and resource-specific and must be tuned accordingly.

%
%
%
%
%

\section{Analysis of the sketching methods}
\label{sec:analysis}

In this section, we provide error analysis for the previously proposed sketching methods. 
Our analysis applies only to a single step of the appropriate optimization algorithm and focuses on two different quantities:
\begin{enumerate}
    \item Error in the GN step: When the randomized sketch is used in SketchSolv, we bound the error in the GN step at the current iterate. 
    \item Condition number: When the randomized sketch is used in SketchPrec, we bound the 2-norm condition number of the preconditioned operator. 
    When this is used as a preconditioner for CG, this result can be combined with the analysis for the reduction in the error in the $\|\cdot\|_{\MB{H}}$-norm with iterations~\cite[Theorem 38.5]{trefethen1997numerical}. 
\end{enumerate}
In this section, we first set the notation for the analysis in \Cref{ssec:setup_analysis} and then summarize the results in \Cref{ssec:analysis_sum}. 
Then, we prove the results with structural and expectational analyses in \Cref{ssec:struct,ssec:prob} respectively.
In addition to the above, condition number estimates for adaptivity are discussed in \Cref{ssec:cdest}.

\subsection{Setup for the analysis}
\label{ssec:setup_analysis}

We first establish some notation. Recall that we have defined $\B{A} = \B{\gampr}^{1/2}$ and we assume it has the compact SVD $\B{A} = \B{U\Sigma V}^\top$.  
We assume that the number of measurements $m\leq n$.
The eigenvalues of the {data-misfit Hessian} $\B{H}$ are 
\begin{equation}
    \lambda_1 = \sigma_1^2, \cdots , \lambda_m = \sigma_m^2, \lambda_{m+1} = \dots =\lambda_{n} = 0.    
\end{equation}
Partition the SVD of $\B{A}$ as 
\begin{equation}
    \B{A} = \bmat{\B{U}_1 & \B{U}_2} \bmat{\B\Sigma_1 \\ & \B\Sigma_2} \bmat{\B{V}_1^\top \\ \B{V}_2^\top},    
\end{equation}
where $\B\Sigma_1 \in \mb{R}^{r\times r}$, $\B{U}_1 \in \mb{R}^{m\times r}$, and $\B{V}_1 \in \mb{R}^{n\times r}$. 
Here, $r \leq m$ is the target rank.

We will need the norm $\|\B{x}\|_{\B{M}} = \|\B{M}^{1/2}\B{x}\|_2$ which is a valid vector norm for any symmetric positive definite matrix $\B{M}$. 
When $\B{M}$ is symmetric positive semidefinite $\|\B{M}^{1/2}\|_2 = \|\B{M}\|_2^{1/2}$. 
This satisfies the inequality 
\begin{equation}
    \frac{\|\B{x}\|_{2}}{\sqrt{\|\B{M}^{-1}\|_2}} \leq \|\B{x}\|_{\B{M}} \leq \sqrt{\|\B{M}\|_2} \|\B{x}\|_2.    
\end{equation}
For an invertible matrix $\B{M}$, we define the 2-norm condition number for inversion as $\kappa_2(\B{M}) = \|\B{M}\|_2 \|\B{M}^{-1}\|_2$. 

We say two symmetric matrices $\B{C},\B{D} \in \R^{n\times n}$ satisfy $\B{C} \preceq \B{D}$ (equivalently, $\B{D} \succeq \B{C}$) if $\B{D} - \B{C}$ is positive semidefinite. 
For any $\B{X} \in \R^{m\times n}$, we have $\B{XCX}^\top \preceq \B{XDX}^\top$ and for any symmetric $\B{M} \in \R^{n\times n}$, $\B{C} + \B{M} \preceq \B{D} +\B{M}$. 
If additionally both $\B{C}, \B{D}$ are positive definite, then $\B{D}^{-1} \preceq \B{C}^{-1}$.

\subsection{Summary of main results}
\label{ssec:analysis_sum}

In this section, we assume that $\lambda_{r+1} \neq 0$. Let us define 
\begin{equation*}
    \text{id}_r \equiv \frac{\sum_{j>r} \lambda_j}{\lambda_{r+1}}
\end{equation*}
as the {\em intrinsic dimension} of the tail of the eigenvalues. 
This is also the stable rank of the matrix $\B\Sigma_2$. 
This quantity is important in our analysis.  

In many applications, the eigenvalues of the operator $\B{H}$ decay rapidly, so that the intrinsic dimension $\id_r$ is often bound independently of the dimensions of the operator. 
In analogy with the decay of the singular values of the Fredholm integral equation in deterministic inverse problems~\cite[Section 1.2.2]{hansen1998rank}, we consider two cases: 
\begin{enumerate}
    \item \textbf{Moderately ill-posed}: We say that the problem is moderately ill-posed if $\lambda_j = C_m j^{-\alpha}$ where $C_m$ is a constant, $\alpha > 1$. 
    Then the intrinsic dimension is bounded by 
    \begin{equation*}
        \id_r \leq \frac{1}{\lambda_{r+1}}\sum_{j>r}^\infty \lambda_j = \zeta(\alpha),
    \end{equation*}
    where $\zeta$ is the Riemann zeta function. 
    If $\alpha \leq 1$, the problem is said to be mildly ill-posed, which we do not consider, since the tail is not summable.
    \item \textbf{Severely ill-posed}: We say that the  problem is severely ill-posed if $\lambda_j = C_s \exp(-\alpha j)$ where $\alpha > 0$ and $C_s$ is a constant.
    \begin{equation*}
        \id_r \leq \frac{1}{\lambda_{r+1}}\sum_{j>r}^\infty \lambda_j = \frac{1}{(1-e^{-\alpha})}.
    \end{equation*}
\end{enumerate}
The bounds that we present below all involve the intrinsic dimension $\id_r$ and do not show explicit dependence on the dimension of the matrices. 
To illustrate the intrinsic dimension, consider the case $n-r=1024$; we plot the intrinsic dimension as a function of $\alpha$ for the moderately and severely ill-posed cases in Figure~\ref{fig:idr}. 
As can be seen, the intrinsic dimension decreases with the decay parameter $\alpha$ in both cases. 
Furthermore, this dimension is much smaller than the number of grid points $n$.
\begin{figure}[!ht]
    \centering
    \begin{subfigure}[b]{0.49\textwidth}
        \centering
        \includegraphics[width=\textwidth]{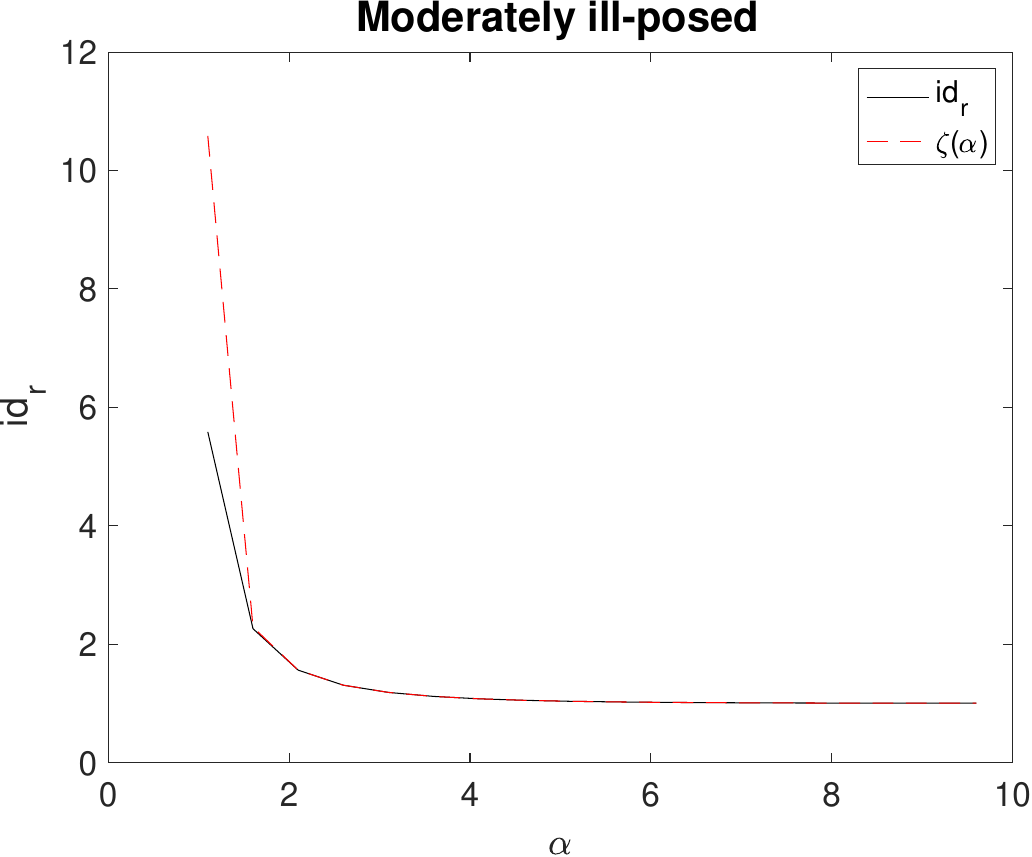}
    \end{subfigure}
    \hfill
    \begin{subfigure}[b]{0.49\textwidth}
        \centering
        \includegraphics[width=\textwidth]{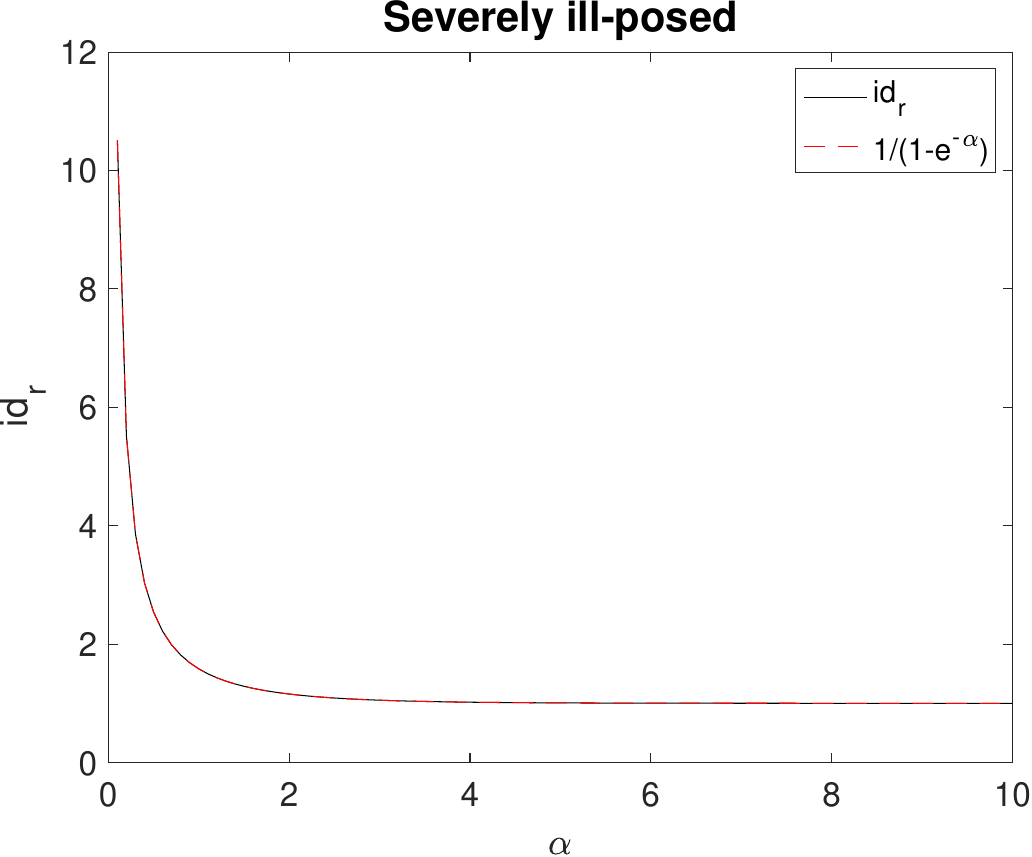}
    \end{subfigure}
    \caption{Plot of intrinsic dimension and the bounds for (left) moderately and (right) severely ill-posed cases.}
    \label{fig:idr}
\end{figure}

We first derive bounds for the RandSVD algorithm. 
\begin{theorem}[RandSVD]\label{thm:randsvd}
    Let $\B\Omega \in \mb{R}^{n\times \ell}$ be a standard Gaussian random matrix which is an input to \Cref{alg:randsvd}, where $\ell = r + p$, $r$ is the target rank and $p \geq 2$ is an oversampling parameter. 
    Consider the output of \Cref{alg:randsvd}, $\Bh{H}$, and define $\MBh{H} = \B\gampr^{-1/2}(\B{I}+\Bh{H})\B\gampr^{-1/2}$. 
    Let 
    \begin{equation*}
        \Psi_{\B{H}}(r,p) \equiv \lambda_{r+1}+ \frac{r}{p - 1} \sum_{j > r}\lambda_j = \left(1 + \frac{r}{p - 1}\id_r \right) \lambda_{r+1}.
    \end{equation*} 
    We have the following bounds: 
    \begin{enumerate}
        \item The error in the {SketchSolv} solution (defined as \cref{eqn:dxsol}) satisfies
        \begin{equation*}
            \expect{ \frac{\|\delta\B{x} - \delta\Bh{x}\|_{\B\gampr^{-1}}}{\|\delta\B{x}\|_{\B\gampr^{-1}}} } \leq  \Psi_{\B{H}}(r,p).
        \end{equation*}
        \item  The condition number of the preconditioned operator by the SketchPrec method satisfies
        %
        \begin{equation*}
            \expect{ \kappa_2(( \B{I} + \Bh{H})^{-1/2}(\B{I} + \B{H})(\B{I} + \Bh{H})^{-1/2})    } \leq 1 + \Psi_{\B{H}}(r,p) .
        \end{equation*}
        
        %
        %
    \end{enumerate}
\end{theorem}
\begin{proof}
    See Section~\ref{ssec:prob}.
\end{proof}
\paragraph{Number of samples} 
We now discuss how to select the number of samples $\ell$. 
Assume that $r_\epsilon$ is the largest index such that $\lambda_{j} \leq \epsilon$ for all $j > r_\epsilon$, and where $\epsilon$ is a user-specified threshold. 
If we take $\ell = \lceil\frac{6}{5}r_\epsilon + 1 \rceil$, then 
\begin{equation*}
    \Psi_{\B{H}}(r,p) \leq \left( 1 + 5 \text{id}_{r_\epsilon}\right)\epsilon.
\end{equation*}
We can then readily bound the appropriate quantities in Theorem~\ref{thm:randsvd}. 
For example,
\begin{equation*}
    \expect{\kappa_2(( \B{I} + \Bh{H})^{-1/2}(\B{I} + \B{H})(\B{I} + \Bh{H})^{-1/2})}  \leq 1 + \left( 1 + 5 \id_{r_\epsilon}\right)\epsilon.    
\end{equation*}
With this simplified representation, we can now interpret the results of \Cref{thm:randsvd}.
The condition number of the preconditioned system is close to 1, provided that the 
eigenvalues of $\B{H}$ are decaying rapidly (either moderately or severely ill-posed).
Similarly, the error in the GN step is small in the $\|\cdot\|_{\B\gampr^{-1}}$ norm, under the same conditions. 

{The above discussion is especially useful in scenarios where additional information about the eigenvalues of the underlying GN Hessian is known. However, in certain scenarios such as in operational data assimilation, this might not be known. In such cases, we have to treat the number of samples as a tuning parameter.}

We now provide analysis for the Nystr\"om sketch. 
\begin{theorem}[Nystr\"om]\label{thm:nystrom_solv}
Let $\B\Omega \in \mb{R}^{n\times \ell}$ be a standard Gaussian random matrix which is an input to Algorithm~\ref{alg:nystrom}, where $\ell = r + p$, $r$ is the target rank and $p \geq 2$ is an oversampling parameter. 
Consider the output of Algorithm~\ref{alg:nystrom}, $\Bh{H}$, and define $\MBh{H} = \B\gampr^{-1/2}(\B{I}+\Bh{H})\B\gampr^{-1/2}$. 
Then, with $\Psi_{\B{H}}(r,p)$ as defined in Theorem~\ref{thm:randsvd}
\begin{enumerate}
    \item The error in the {SketchSolv} solution satisfies
    \begin{equation*}
        \expect{ \frac{\|\delta\B{x} - \delta\Bh{x}\|_{\B\gampr^{-1}}}{\|\delta\B{x}\|_{\B\gampr^{-1}}}} \leq \Psi_{\B{H}}(r,p).
    \end{equation*}
    \item The condition number of the preconditioned operator by the SketchPrec method satisfies
    \begin{equation*}
        \expect{ {\kappa_2(( \B{I} + \Bh{H})^{-1/2}(\B{I} + \B{H})(\B{I} + \Bh{H})^{-1/2})}} \leq 1 +   \Psi_{\B{H}}(r,p).
    \end{equation*}
\end{enumerate}
\end{theorem}
\begin{proof}
    See Section~\ref{ssec:prob}.
\end{proof}
Surprisingly, both RandSVD and Nystr\"om approximations give identical upper bounds in expectation. The interpretation of Nystr\"om is, therefore, similar to RandSVD. Both methods are also comparable in computational cost. 

The result (Theorem~\ref{thm:nystrom_solv}) is slightly weaker compared to the result in~\cite[Theorem 5.1]{frangella2021randomized}. From the proof of~\cite[Theorem 5.1]{frangella2021randomized}, we find the inequality (in our notation, with $p=r+1$ and $\ell = 2r-1$)
\[\expect{ {\kappa_2(( \B{I} + \Bh{H})^{-1/2}(\B{I} + \B{H})(\B{I} + \Bh{H})^{-1/2})}} \leq 2 + \left(3 + \frac{4e^2}{r+1} \id_{r}\right)\lambda_{r+1}.\]
In contrast, our bound simplifies to (assuming $r > 2$)
\[ \expect{ {\kappa_2(( \B{I} + \Bh{H})^{-1/2}(\B{I} + \B{H})(\B{I} + \Bh{H})^{-1/2})}} \leq 1 + \left(1 + \frac{r}{r-2} \id_{r}\right)\lambda_{r+1}.\] 
The upper bound in our result is slightly worse for larger $r$. However, our proof technique applies to other sketching methods. 

Finally, we turn to the analysis of the SingleView sketch.

\begin{theorem}[SingleView]\label{thm:singleview}
Suppose $\B\Omega \in \mb{R}^{n\times \ell_1}$ and $\B\Psi \in \mb{R}^{m\times \ell_2}$ (standard Gaussian random matrices) are inputs to \Cref{alg:singleview}, where $\ell_1 = 2r + 1, \ell_2 = 2\ell_1 + 1 = 4r + 3$ and $r$ is the target rank. 
Consider the output of \Cref{alg:singleview}, $\Bh{H}$, and define $\MBh{H} = \B\gampr^{-1/2}(\B{I}+\Bh{H})\B\gampr^{-1/2}$.
Let 
\begin{equation*}
    {\Theta_{\B{H}}(r,p) \equiv  4\left( \sum_{j>r} \lambda_1\lambda_j  \right)^{1/2} +  4  \left(\sum_{j>r}\lambda_j \right) =  \left( \left(\frac{\lambda_1}{\lambda_{r+1}}\right)^{1/2} \id_r^{1/2} +  \id_r \right) 4 \lambda_{r+1}. }  
\end{equation*}
Then,
\begin{enumerate}
    \item The error in the {SketchSolv} solution satisfies
    \begin{equation*}
        \expect{ \frac{\|\delta\B{x} - \delta\Bh{x}\|_{\B\gampr^{-1}}}{\|\delta\B{x}\|_{\B\gampr^{-1}} }} \leq  \Theta_{\B{H}}(r,p).
    \end{equation*}
    \item The condition number of the preconditioned operator by the SketchPrec method satisfies 
    \begin{equation*}
        \expect{\sqrt{\kappa_2(( \B{I} + \Bh{H})^{-1/2}(\B{I} + \B{H})(\B{I} + \Bh{H})^{-1/2})}}    \leq 1 +  \Theta_{\B{H}}(r,p).        
    \end{equation*}
    %
    \end{enumerate}
\end{theorem}
\begin{proof}
    See Section~\ref{ssec:prob}.
\end{proof}

The results of SingleView are considerably weaker than RandSVD and Nystr\"om. 
We see the appearance of $\lambda_1$ in the bounds, which can be large. 
Furthermore, we are only able to bound the expected value of the square root of the condition number, rather than the expected value of the condition number. 
We explain these reasons in \Cref{ssec:prob}. 

\subsection{Structural analysis} 
\label{ssec:struct}

In the analysis below, the error $\B{E} = \B{H} - \Bh{H}$ plays an important role. 
We consider two different cases: $\B{E} \succeq \B{0}$, and otherwise. 
This distinction is important because, with the stronger assumption $\B{E} \succeq \B{0}$, we can obtain tighter bounds in some cases. This assumption is satisfied for two of the methods (RandSVD and Nystr\"om) as we will show.
We call this analysis {\em structural} because it does not take into account the randomness of the inputs. 
It applies to any approximation $\B{H} \approx \Bh{H}$, that need not be a randomized or even a low-rank approximation. 

Our first result quantifies the error in the solution of the linear system. 

 \begin{theorem}[Error in solve]\label{thm:solve}
     Let $\B{E} =  \B{H} - \Bh{H}$ denote the error in the low-rank approximation. The relative error in the linear system satisfies
     \begin{equation*}
         \frac{\|\delta\B{x} - \delta\Bh{x}\|_{\B\gampr^{-1}}}{\|\delta\B{x}\|_{\B\gampr^{-1}}} \leq  \|\B{E}\|_2.
     \end{equation*} 
 \end{theorem}
 This result has similarities to~\cite[Proposition 3.1]{frangella2021randomized}, but applies to a more general covariance matrix. 
 The interpretation of this bound is that the error is small if $\|\B{E}\|_2$ is small. 
 We were not able to obtain a tighter bound if, additionally, $\B{E} \succeq \B{0}$ holds.
 The analysis bounds the error in the $\|\cdot \|_{\B\gampr^{-1}}$ norm; if, however, the error is desired in the 2-norm, we can equivalently write 
 \begin{equation*}
     \frac{\|\delta\B{x} - \delta\Bh{x}\|_{2}}{\|\delta\B{x}\|_{2}} \leq \sqrt{\kappa_2(\B\gampr)}\frac{\|\delta\B{x} - \delta\Bh{x}\|_{\B\gampr^{-1}}} {\|\delta\B{x}\|_{\B\gampr^{-1}}} \leq  \sqrt{\kappa_2(\B\gampr)}\|\B{E}\|_2.    
\end{equation*}

 \begin{proof}[Proof of Theorem~\ref{thm:solve}]
 Observe that 
\begin{equation*}
    \MB{H} = \B\gampr^{-1/2} ( \B{I} + \Bh{H} + \B{E}) \B\gampr^{-1/2} = \MBh{H} + \B\gampr^{-1/2}  \B{E} \B\gampr^{-1/2} .   
\end{equation*}
 From $\MBh{H}\delta\Bh{x} = -\B{g} = \MB{H}\delta\B{x}$, we can write 
 \begin{equation*}
     \MBh{H}(\delta\Bh{x} - \delta\B{x}) = (\MB{H}- \MBh{H})\delta\B{x} =  \B\gampr^{-1/2}  \B{E} \B\gampr^{-1/2}\delta\B{x}.   
 \end{equation*}
 On the left hand side, we expand $\MBh{H}$, we get 
 \begin{equation*}
     \B\gampr^{-1/2}(\B{I} + \Bh{H} )\B\gampr^{-1/2} (\delta\Bh{x} - \delta\B{x}) =  \B\gampr^{-1/2}  \B{E} \B\gampr^{-1/2}\delta\B{x} ,
 \end{equation*} 
 which simplifies to $\B\gampr^{-1/2} (\delta\Bh{x} - \delta\B{x}) =  (\B{I} + \Bh{H} )^{-1}\B{E} \B\gampr^{-1/2}\delta\B{x}$. 
 Using submultiplicativity
 \begin{equation*}
     \|(\delta\Bh{x} - \delta\B{x})\|_{\B\gampr^{-1} } \leq \|(\B{I} + \Bh{H} )^{-1}\|_2  \| \B{E}\|_2 \|\delta\B{x}\|_{\B\gampr^{-1}}.
 \end{equation*}
 The proof is complete by noting that the singular values of $(\B{I} + \Bh{H})$ are at least $1$. \qed
 \end{proof}

We now analyze the condition number of the preconditioned matrix. In this analysis, we assume that $\Bh{H}\preceq \B{H}$. This assumption holds for the Randomized SVD and the Nystr\"om approximation, but not for the Single View approach. 
 \begin{theorem}[Condition number bound]\label{thm:bound_loew}
Let $\B{E} = \B{H} - \Bh{H}$.
Then 
\begin{equation*}
        \kappa_2\left( \left(\B{I} + \Bh{H}\right)^{-1/2} \left(\B{I} + \B{H}\right)\left(\B{I} + \Bh{H}\right)^{-1/2} \right) \leq \left(1 +  \|\B{E}\|_2\right)^2.        
\end{equation*}
If, additionally, $\B{E} \succeq \B{0}$, then 
\[ \kappa_2\left( \left(\B{I} + \Bh{H}\right)^{-1/2} \left(\B{I} + \B{H}\right)\left(\B{I} + \Bh{H}\right)^{-1/2} \right) \le 1 + \|\B{E}\|_2.\]
\end{theorem}

It is clear from the theorem that the bound is tighter when $\B{E} \succeq \B{0}$. The interpretation of this theorem is similar to the others in this section.

\begin{proof}
By definition, the condition number is 
\begin{align*}
      \kappa_2\left( \left(\B{I} + \Bh{H}\right)^{-1/2} \left(\B{I} + \B{H}\right)\left(\B{I} + \Bh{H}\right)^{-1/2} \right) &= \underbrace{\|\left(\B{I} + \Bh{H}\right)^{-1/2} \left(\B{I} + \B{H}\right)\left(\B{I} + \Bh{H}\right)^{-1/2}\|_2}_{\alpha} \times \nonumber\\
      &\underbrace{\| \left(\B{I} + \Bh{H}\right)^{1/2} \left(\B{I} + \B{H}\right)^{-1}\left(\B{I} + \Bh{H}\right)^{1/2}\|_2}_\beta.    
\end{align*}
We bound the two terms separately.
To bound $\alpha$, we write $\B{I} + \B{H} = \B{I} +\B{H} - \Bh{H} + \Bh{H}$, so that 
\begin{align*}
        \left(\B{I} + \Bh{H}\right)^{-1/2} \left(\B{I} + \B{H}\right)\left(\B{I} + \Bh{H}\right)^{-1/2} &= \left(\B{I} + \Bh{H}\right)^{-1/2} \left(\B{I} + \Bh {H}\right)\left(\B{I} + \Bh{H}\right)^{-1/2} + \nonumber\\ 
        & \left(\B{I} + \Bh{H}\right)^{-1/2} \left(\B{H} - \Bh {H}\right)\left(\B{I} + \Bh{H}\right)^{-1/2}\,.
\end{align*}
Therefore, using the triangle inequality $\alpha \leq 1 + \|(\B{I} +\Bh{H})^{-1/2} (\B{H}-\Bh{H}) (\B{I} + \Bh{H})^{-1/2}\|_2$. 
Let $\B{C},\B{D}$ be two matrices such that the product $\B{CD}$ is normal, then using~\cite[Proposition IX.1.1]{bhatia1997matrix}, we have
$\|\B{CD}\|_2 \le \|\B{DC}\|_2$. We take $\B{C} = (\B{I} +\Bh{H})^{-1/2} $, and $\B{D} = (\B{H}-\Bh{H}) (\B{I} + \Bh{H})^{-1/2}$ and observe that $\B{CD} = (\B{I} +\Bh{H})^{-1/2} (\B{H}-\Bh{H}) (\B{I} + \Bh{H})^{-1/2}$ is symmetric (therefore, normal) to obtain 
\begin{align*}
\|\left(\B{I} + \Bh{H}\right)^{-1/2} \left(\B{H} - \Bh {H}\right)\left(\B{I} + \Bh{H}\right)^{-1/2}\|_2 \leq \|(\B{H} - \Bh{H} )(\B{I} + \Bh{H})^{-1}\|_2\,.
\end{align*}
Observe that the eigenvalues of $\B{I} + \Bh{H}$ is bounded from below by 1 and hence 
\begin{align*}
\alpha \leq 1 + \|\B{E}\|_2\,.
\end{align*}
To bound $\beta$,  we split into {two cases:
\begin{enumerate}
    \item For the general case, we once again use ~\cite[Proposition IX.1.1]{bhatia1997matrix}. We take $\B{C} = (\B{I} +\Bh{H})^{1/2} (\B{I} + \B{H})^{-1}$ and $\B{D} = (\B{I} +\Bh{H})^{1/2}$. We have
\begin{align*}
\| \left(\B{I} + \Bh{H}\right)^{1/2} \left(\B{I} + \B{H}\right)^{-1}\left(\B{I} + \Bh{H}\right)^{1/2}\|_2 \le \|\left(\B{I} + \Bh{H}\right)  \left(\B{I} + \B{H}\right)^{-1} \|_2 \,.
\end{align*}
Writing $\B{I} + \Bh{H} = \B{I} + \B{H} + \Bh{H} - \B{H}$ and using triangular inequality, we obtain
\begin{align*}
\beta \le 1 + \|(\Bh{H} - \B{H})(\B{I} + \B{H})^{-1}\|_2 \le 1 +  \|(\Bh{H} - \B{H})\|_2 \|(\B{I} + \B{H})^{-1}\|_2
\end{align*}
Observe that the eigenvalues of $\B{I} + {\B{H}}$ is bounded from below by $1$ and hence 
\begin{align*}
\beta \leq 1 + \|\B{E}\|_2\,.
\end{align*}
\item The special case $\B{E} \succeq \B{0}$ implies  $\Bh{H} \preceq \B{H}$ and, therefore, $\B{I} + \Bh{H} \preceq \B{I} + \B{H}$. By~\cite[Lemma V.1.7]{bhatia1997matrix}, $\|(\B{I} + \Bh{H})^{1/2}(\B{I} + \B{H})^{-1/2}\|_2 \leq 1$ which implies $\beta \le 1$.
\end{enumerate}  }
We have the desired result by combining the bounds for $\alpha$ and $\beta$. \qed
\end{proof}

\subsection{Analysis in expectation}
\label{ssec:prob}

In this subsection, we use the structural bounds to derive bounds in expectation for the three sketching methods from Section~\ref{ssec:sketching}. 
\begin{proof}[Proof of Theorem~\ref{thm:randsvd}]

Compute the sketch $\B{Y} = \B{A\Omega}$ and the thin QR factorization $\B{Y} = \B{QR}$. 
The low-rank approximation is $\Bh{H} = \B{A}^\top\B{QQ}^\top\B{A} = \Bh{V}\Bh\Lambda\Bh{V}\t$. We tackle each part individually.

Part 1. By Theorem~\ref{thm:solve}, the error in the solution depends only on the error in the low-rank approximation $\|\B{H}-\Bh{H}\|_2$. 
Using the fact that the orthogonal projector is symmetric and idempotent, we can rewrite this as 
\begin{equation*}
    \|\B{H}-\Bh{H}\|_2 = \|\B{A}^\top(\B{I}-\B{QQ}^\top)\B{A}\|_2 = \|(\B{I}-\B{QQ}^\top)\B{A}\|_2^2.    
\end{equation*}
Since $\range(\B{Y}) \subset \range(\B{Q})$, by~\cite[Proposition 8.5 and Theorem 9.1]{halko2011finding}, 
\begin{equation*}
    \|(\B{I}-\B{QQ}^\top)\B{A}\|_2^2 \leq \|(\B{I}-\B{YY}^\dagger)\B{A}\|_2^2 \leq \|\B\Sigma_2\|_2^2 + \|\B\Sigma_2\B\Omega_2\B\Omega_1^\dagger\|_2^2,    
\end{equation*}
{where $\B\Omega_1 = \B{V}_1\t \B\Omega$ and $\B\Omega_2 = \B{V}_2\t\B\Omega$. }
Taking expectations and using the fact that the spectral norm is dominated by the Frobenius norm
\begin{equation*}
    \expect{ \|\B{H}-\Bh{H}\|_2 } \leq  \> \|\B\Sigma_2\|_2^2 + \expect{\|\B\Sigma_2\B\Omega_2\B\Omega_1^\dagger\|_F^2}.    
\end{equation*}
From the proof of~\cite[Theorem 10.5]{halko2011finding}, $\expect{\|\B\Sigma_2\B\Omega_2\B\Omega_1^\dagger\|_F^2} \leq \frac{r}{p-1}\|\B\Sigma_2\|_F^2$. 
Plug the inequality into the above bound, and identify the singular values with the eigenvalues. 

Part 2. Since $\B{QQ}\t$ is an orthogonal projector, $\B{QQ}\t \preceq\B{I} $ and $\Bh{H} \preceq\B{H}$. The special case of Theorem~\ref{thm:bound_loew} applies and $  \kappa_2(( \B{I} + \Bh{H})^{-1/2}(\B{I} + \B{H})(\B{I} + \Bh{H})^{-1/2})     \leq 1 + \|\B{E}\|_2$. Taking expectations and applying the result from Part 1 gives us the desired bounds. \qed
\end{proof}

\begin{proof}[Proof of Theorem~\ref{thm:nystrom_solv}]
As before, we tackle each part separately.

Part 1. Follows from Theorem~\ref{thm:solve} and~\cite[Theorem 14.1]{martinsson2020randomized}.

Part 2. First, observe that 
\begin{equation*}
    \Bh{H} = \B{Y} (\B\Omega^\top\B{Y})^\dagger \B{Y}^\top = \B{H\Omega} (\B\Omega^\top\B{H\Omega})^\dagger \B{H\Omega}^\top = \B{H}^{1/2}\B{P}_{\B{H}^{1/2}\B\Omega} \B{H}^{1/2},
\end{equation*}
where $\B{P}_{\B{H}^{1/2}\B\Omega} = \B{H}^{1/2}\B{\Omega} (\B\Omega^\top\B{H\Omega})^\dagger \B{\Omega}^\top\B{H}^{1/2}$ is an orthogonal projector. 
This means $\B{P}_{\B{H}^{1/2}\B\Omega}\preceq \B{I}$ and $\Bh{H} \preceq \B{H}$. 
Therefore, the second part of Theorem~\ref{thm:bound_loew} applies and 
\begin{equation*}
    \expect{\kappa_2(( \B{I} + \Bh{H})^{-1/2}(\B{I} + \B{H})(\B{I} + \Bh{H})^{-1/2})} \leq 1 + \expect{ \|\B{E}\|_2}.    
\end{equation*}
The proof then follows from~\cite[Theorem 14.1]{martinsson2020randomized}. \qed
%
\end{proof}

\begin{proof}[Proof of Theorem~\ref{thm:singleview}] 
To prove this theorem, we first recognize that $\B{E} = \B{H}-\Bh{H}$ is no longer positive semidefinite, so we have to use the weaker result in \Cref{thm:bound_loew}. 

Part 1. By \Cref{thm:solve}, the error in the solution depends only on the error in the low-rank approximation $\|\B{E}\|_2$. 
Using triangle inequality and submultiplicativity
\begin{align*}
    \|\B{H}-\Bh{H}\|_2 = &\>  \|\B{A}^\top\B{A} - \B{A}^\top \Bh{A} + \B{A}^\top \Bh{A} - \Bh{A}^\top \Bh{A}\|_2 \\
    \leq &  \> \| \B{A}\|_2 \|\B{A}-\Bh{A}\|_2 + \|\B{A}-\Bh{A}\|_2 \|\Bh{A}\|_2. 
\end{align*}
Write $\Bh{A} = \Bh{A} - \B{A} + \B{A}$, and once again apply triangle inequality to get
\begin{equation*}
    \|\B{H}-\Bh{H}\|_2 \leq 2\|\B{A}\|_2 \|\B{A}-\Bh{A}\|_2 + \|\B{A}-\Bh{A}\|_2^2.    
\end{equation*}
The spectral norm is bounded by the Frobenius norm. 
Take expectations and apply Cauchy-Schwartz inequality to get 
\begin{equation}\label{eqn:singleviewinter} 
    \expect{\|\B{H}-\Bh{H}\|_2 } \leq 2\|\B{A}\|_2\left(\expect{ \|\B{A}-\Bh{A}\|_F^2}\right)^{1/2} + \expect{\|\B{A}-\Bh{A}\|_F^2}.
\end{equation}
By~\cite[Theorem 4.3]{tropp2017practical},
\begin{equation*}
    \expect{\|\B{A}-\Bh{A}\|_F^2} \leq (1 + f(\ell_1,\ell_2))(1 + f(r, \ell_1) ) \sum_{j > r}\sigma_j^2,
\end{equation*}
where $f(a,b) = a/(b-a-1)$. 
Take $\ell_1 = 2r + 1$, and $\ell_2 = 2\ell_1 + 1$ so that {$f(r, \ell_1) = f(\ell_1, \ell_2)  = 1$. 
This gives $\expect{\|\B{A}-\Bh{A}\|_F^2} \leq 4 \sum_{j > r}\sigma_j^2$.}
Plug into~\cref{eqn:singleviewinter} and simplify.

Part 2: By Theorem~\ref{thm:bound_loew},
\begin{equation*}
    \sqrt{\kappa_2((\B{I} + \Bh{H})^{-1/2}(\B{I} + \B{H})(\B{I} + \Bh{H})^{-1/2})}  \leq 1 + \|\B{E}\|_2.    
\end{equation*}
The rest of the proof is similar to Part 1. \qed
%
\end{proof}

\subsection{Condition number estimate}\label{ssec:cdest}

Here, we show that the quantity $\kappa_{\rm sk}$ discussed in \Cref{subsubsec:adap}, and in \cref{eqn:condtext2} is a 2-norm estimator of the condition number of the GN system discussed in \Cref{subsubsec:skprec}.
{\begin{proposition}\label{prop:condest}
    Let $\B{v}$ be a random vector uniformly distributed on the sphere $S^{n-1}$. Let $\kappa_{\rm sk} = \| (\B{I}+\B{H})(\B{I}+\Bh{H})^{-1}\B{v}\|_2$. Let $\B{E}\equiv \B{H}-\Bh{H} \succeq \B{0}$. Then, with probability at least $1- 0.8\theta^{-1/2}n^{1/2}$
    \[  \kappa_2((\B{I}+\Bh{H})^{-1/2}(\B{I}+\B{H})(\B{I}+\Bh{H})^{-1/2}) \le \sqrt{\theta} \kappa_{\rm sk}.    \]
\end{proposition}
\begin{proof}
From the proof of Theorem~\ref{thm:bound_loew}, we have $ \|(\B{I}+\Bh{H})^{1/2}(\B{I}+\B{H})^{-1}(\B{I}+\Bh{H})^{1/2}\|_2 \le 1$, so it is sufficient to find a bound for $ \|(\B{I}+\Bh{H})^{-1/2}(\B{I}+\B{H})(\B{I}+\Bh{H})^{-1/2}\|_2$. As in the proof of Theorem~\ref{thm:bound_loew}, 
\[ \|(\B{I}+\Bh{H})^{-1/2}(\B{I}+\B{H})(\B{I}+\Bh{H})^{-1/2}\|_2 \le \|(\B{I}+\B{H})(\B{I}+\Bh{H})^{-1}\|_2, \]
so now 
\[ \kappa_2((\B{I}+\Bh{H})^{-1/2}(\B{I}+\B{H})(\B{I}+\Bh{H})^{-1/2}) \le \|(\B{I}+\B{H})(\B{I}+\Bh{H})^{-1}\|_2.\] 
Apply~\cite[Theorem 1 and corollary]{dixon1983estimating} with $\B{A} =(\B{I}+\Bh{H})^{-1}(\B{I}+\B{H})^2(\B{I}+\Bh{H})^{-1} $ and $k=1$ to get
\[ \kappa_2((\B{I}+\Bh{H})^{-1/2}(\B{I}+\B{H})(\B{I}+\Bh{H})^{-1/2}) ^2 \le \theta \| (\B{I}+\B{H})(\B{I}+\Bh{H})^{-1}\B{v}\|_2^2,\]
with probability at least $1-0.8\theta^{-1/2}n^{1/2}$. Take square roots to complete the proof. 
\qed
\end{proof}}
As has been mentioned earlier, the assumption $\B{E} \succeq \B{0}$ is satisfied by the RandSVD and the Nystr\"om approaches but is not satisfied by the SingleView approximation. In this case, Proposition~\ref{prop:condest} is not a condition number estimator but a norm estimator for the matrix $(\B{I}+\Bh{H})^{-1/2}(\B{I}+\Bh{H})(\B{I}+\Bh{H})^{-1/2}$. This result provides some justification for the condition number estimate $\kappa_{\rm sk}$ defined in~\eqref{eqn:condtext}.

\section{Numerical Experiments}

This section describes experiments on two different PDEs, namely
\begin{enumerate*}[label={(\roman*)}]
    \item the 1-D Burgers equation, and 
    \item barotropic vorticity equation.
\end{enumerate*}
MATLAB 2023a was used to perform the numerical experiments.
We use a custom GN solver with the Mor{\'e}-Thuente line search~\cite[]{more1994line} from POBLANO~\cite[]{dunlavy2010poblano}.
The settings for each experiment are detailed in the corresponding subsections and summarized in \Cref{tab:expt0_settings} respectively.
\begin{table}[!ht]
    \centering
    \begin{tabular}{ c | c | c }
      Setting & Application 1  & Application 2 \\ \hline
        Domain & $x \in (0, 1)$ & $(x, y) \in (0, 1)\times(-1, 1)$ \\ 
        Parameters & $\nu\in \{10^{-3},10^{-2}, 10^{-1},1\}$ & Re = 200, Ro = 0.0016 \\ 
        State size & $n \in \{199, 399, 599, 799\}$ & $n = 128 \times 257 = 32896$ \\
        $\Delta t$ & 0.01 & 0.001225 (around 6 hours) \\
        $n_t$ & 20 & 8 \\
        $n_{\rm obs}$ & 15 & 256 \\
        $\B\gampr$ & $(0.5\B{I} - 500\B\Delta^*)^{-2}$ & $(-0.06\B\Delta^*)^{-2}$ \\
        $\B{R}_i, 1 \leq i \leq n_t $& $0.01\cdot \B{I}_{15}$ & $361\cdot \B{I}_{256}$\\
        Relative $\|\nabla \mc{J}\|_\infty$ tolerance & $ 10^{-6}$ & $10^{-6}$\\
        PCG tolerance & $10^{-9}$ & $10^{-9}$\\
        \hline
    \end{tabular}
    \caption{Summary of experiment settings. $\B\Delta^*$ refers to the discrete Laplacian operator with Dirichlet boundary conditions: the 3-point and 5-point stencil for Experiments 1 and 2 respectively.   }
    \label{tab:expt0_settings}
\end{table}
We define the relative error of the analysis trajectory as $\frac {\| \B{x}^a_i - \B{x}^{\rm true}_i \|_2}{ \| \B{x}^{\rm true}_i \|_2}$ at each instance of $i \Delta t$ where $ 0 \leq i \leq n_k$. 
Here, $\B{x}^a_i$ is the analysis $\B{x}^a_0$ forecasted to time $t_i$, $\B{x}^{\rm true}_i$ is the true solution at $t_i$, and $n_k$ is a discrete number of time instances we compare the two values. 

Note that we choose $n_k > n_t$, meaning we consider the forecasting potential of the $\B{x}^{a}_0$ without any observations after the observation window has elapsed. 
For the test problems in Experiments 1.1 and 2, we report 
\begin{enumerate*}[label={(\roman*)}]
    \item the total number of PCG iterations that are required to obtain the solution of the SC-4DVAR optimization problem,
    \item the total number of parallelizable (or offline) TLM $\B{A}$ and ADJ $\B{A}\t$ solutions that are required to create sketches (SketchSolv or SketchPrec) across the full optimization,
    \item the total number of sequential (or online) TLM and ADJ solves used in the computation of the gradient (SketchSolv and SketchPrec), GN Hessian (SketchPrec only), and  
    \item the total number of PDE solves including those done inside the line-search.
\end{enumerate*}
The situation without a preconditioner is referred to as $\textrm{Prec}\_\B{\gampr}$ to indicate that it only uses the first-level preconditioning.
The Lanczos algorithm is initialized (as the starting vector) with the right-hand side of the linear system in \cref{eqn:pcgsystem}.
Using a random starting vector also works, but the random choice results in slower PCG convergence.
In this section, $\Bh{H}$ generated by the Lanczos method is loosely referred to as a sketch, and its Krylov subspace dimension is called the sketch size.

\subsection{Application 1: 1-D Burgers' Equation}
\label{ssec:burg}
The first application we consider is the non-linear, one-dimensional viscous Burgers' equation. This equation describes the non-linear advection and diffusion of a fluid in space and time.
Some of the problem settings are inspired by the advection-diffusion example from~\cite{freitag2018low}. 
Consider the PDE
\begin{equation}\label{eqn:1dburgers}
    \frac{\partial u}{\partial t} +  u \frac{\partial u}{\partial x} = \nu \frac{\partial^2 u}{\partial x^2} ,
\end{equation}
with domain $x \in (0, 1)$, diffusion coefficient $\nu = 0.1$, and homogeneous Dirichlet boundary conditions $u(t, 0) = u(t, 1) = 0$. 
The spatial derivatives are computed using second-order central finite differences on a uniform mesh. 
The time stepping uses a 3rd-order, explicit, total variation diminishing Runge Kutta~\cite[]{gottlieb1998tvd} scheme called RK3.
For Experiments 1.1 and 1.2, we use the grid size $n = 399$, giving $\Delta x = 0.0025$, and for Experiments 1.3 and 1.4, we use $n \in \{ 199, 399, 599, 799\}$.
For Experiments 1.1 and 1.2, the RK3 timestep $\Delta t_{ts} = 2.5 \times 10^{-5}$. 
For experiments 1.3 and 1.4, $\Delta t_{ts}$ is chosen sufficiently small based on problem settings to have stable solutions.
The discrete Runge Kutta TLM and ADJ~\cite[]{Sandu_PA2006a} are used to compute the gradients and Hessian vector products. 

\paragraph{Data and Background} For all experiments, the true initial condition $\B{x}^{\rm true}_0$ is set based on the discrete representation of $\sin(\pi x)$.
In standard SC-4DVAR fashion~\cite[]{evensen2022data,freitag2018low,dauvzickaite2021randomised}, the background estimate is chosen by adding a background scaled perturbation to the true initial condition as $\B{x}^b_0 = \B{x}^{\rm true}_0 + \B\gampr^{\frac{1}{2}} \B\xi$, with the covariance matrix $\B\gampr = (0.5\B{I} - 500\B\Delta^*)^{-2}$ where $\B\Delta^*$ is the discrete Laplacian with homogeneous Dirichlet boundary conditions and $\B\xi \sim \mc{N}(\B{0}, \B{I})$.
The parameters of the covariance $\B\gampr$ are hand-tuned.

Observations are recorded for a window of $n_t = 20$ time instances, each instance separated by an interval of $\Delta t = 0.01$ time units.
For all experiments (1.1 through to 1.4), we observe 15 states at the x-intercepts of the observations shown in \cref{fig:expt1-example}.
The observation error for all experiments (1.1 through to 1.4) is $\B{R}_i = 0.01 \cdot \B{I}_{15}$ for $1 \leq i \leq n_t$.
\Cref{fig:expt1-example} shows an example of the true solution, observations at the locations, and a background estimate.
\begin{figure}
    \centering
    \includegraphics[width=0.5\linewidth]{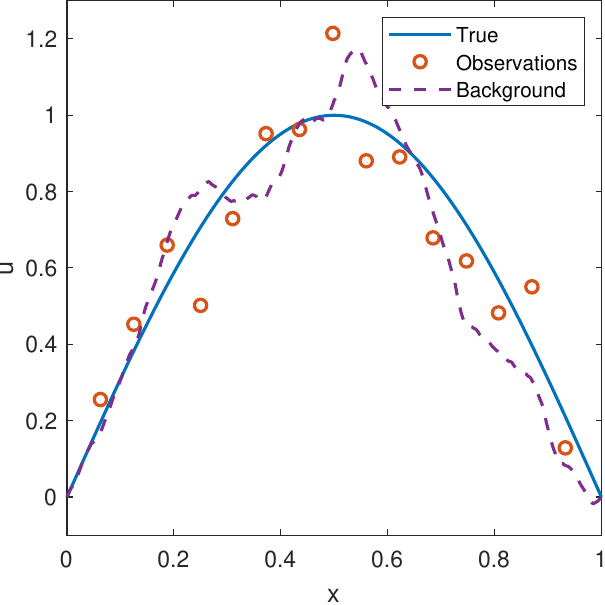}
    \caption{An illustration depicting the true states, its noisy observation at time $0$, and a background estimate.}
    \label{fig:expt1-example}
\end{figure}
\paragraph{Optimization and PCG settings} 
The optimizer for every variant of the GN method is terminated when the norm of the gradient at the $k$-th GN iteration decreases below a specified tolerance relative to the norm of the first gradient (i.e. when $\frac{\|\nabla \mc{J}^{(k)} \|_\infty}{\|\nabla \mc{J}^{(0)} \|_\infty} < 10^{-6}$). 
The optimization is initialized with the background $\B{x}_0^{(0)} = \B{x}^b_0$. PCG for computing the descent direction is terminated when the relative residual of the $k$-th iterate achieves a prespecified tolerance (here $10^{-9}$).

\subsubsection{Experiment 1.1: Comparing the sketching approaches}

We look at the cost and the relative analysis error trajectory with $n_k = 81$ for the following methods
\begin{enumerate*}[label=(\roman*)]
    \item basic : Prec\_$\B{\gampr}$, Prec\_Lanczos, Solv\_Lanczos, 
    \item SketchPrec : RandSVD, Nystr\"om, and SingleView,
    \item SketchSolv : RandSVD, Nystr\"om, and SingleView.
\end{enumerate*}
The sketch sizes are $\ell = 15$ for Lanczos, RandSVD and Nystr\"om, and $\ell_1 = 15, \ell_2 = 31$ for SingleView. The sketches are updated at each GN iteration; we explore the adaptive approach in the next experiment in \Cref{ssec:bve}.
\begin{figure}[!ht]
    \centering
    \begin{subfigure}[b]{0.49\textwidth}
        \centering
        \includegraphics[width=\linewidth]{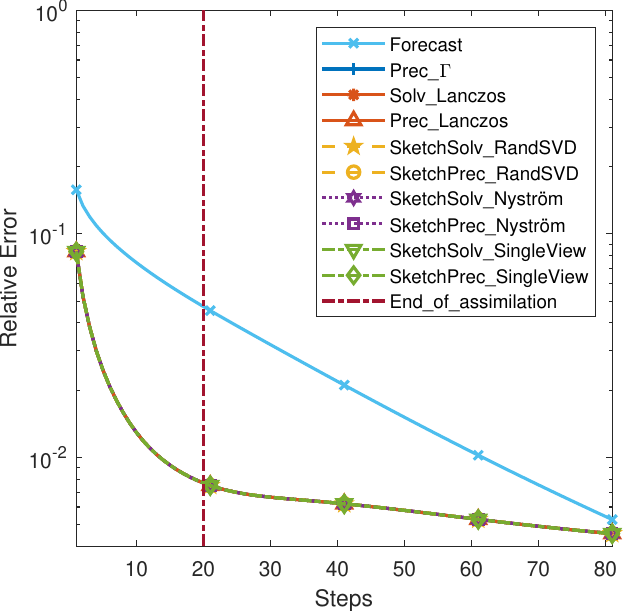}
        \caption{Relative error for 1D-Burgers}
        \label{fig:expt1_rmse}
    \end{subfigure}
    \hfill
    \begin{subfigure}[b]{0.49\textwidth}
        \centering
        \includegraphics[width=\linewidth]{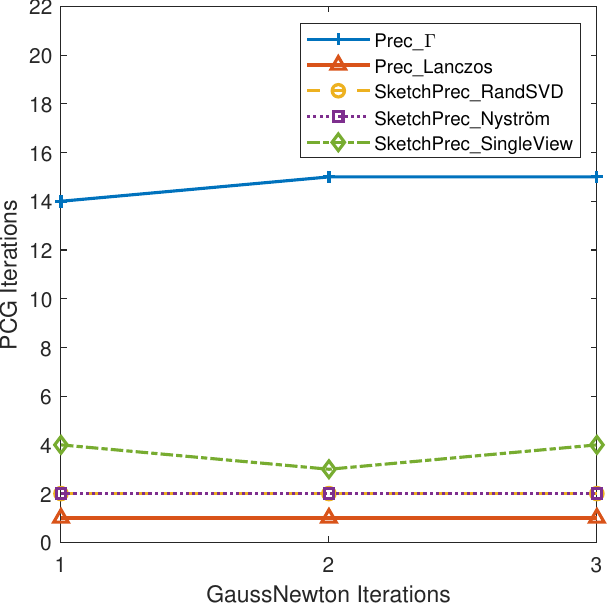}
        \caption{GaussNewton vs PCG iterations.}
        \label{fig:expt1_ncg}
    \end{subfigure}
    \caption{The relative error in \Cref{fig:expt1_rmse} has similar behavior in time across the different methods. \Cref{fig:expt1_ncg} shows a decrease in the number of PCG iterations for the preconditioned methods.}
    \label{fig:expt1_fig1}
\end{figure}
\begin{table}[!ht]
    \centering
    \begin{tabular}{ c | c | c | c | c | c | c}
       Method & \#PCG & \multicolumn{2}{c|}{Par/Offline/Setup} & \multicolumn{3}{c}{Seq/Online} \\ \hline        & \thead{Iterations} & \thead{\#TLM} & \thead{\#ADJ} & \thead{\#FWD} & \thead{\#TLM} & \thead{\#ADJ} \\\hline
       \multicolumn{7}{c}{Basic} \\\hline
        Prec\_$\B{\gampr}$ & 44 & --- & --- & 8 & 44 & 47\\
        Prec\_Lanczos & 3 & --- & --- & 8 & 48 & 51 \\
        Solv\_Lanczos & {---} & --- & --- & 8 & 45 & 48 \\ \hline   
    \multicolumn{7}{c}{SketchPrec} \\\hline  
        RandSVD & 6 & 45 & 45 & 8 & 6 & 9 \\
        Nystr\"om & 6 & 45 & 45 & 8 & 6 & 9 \\
        SingleView & 11 & 45 & 93 & 8 & 11 & 14 \\\hline
    \multicolumn{7}{c}{SketchSolv} \\\hline
        RandSVD & {---} & 45 & 45 & 8 & 0 & 3 \\
        Nystr\"om & {---} & 45 & 45 & 8 & 0 & 3  \\
        SingleView & {---} & 45 & 93 & 10 & 0 & 3 
    \end{tabular}
    \caption{The comparison of proposed methods for the 1-D Burgers equation. All methods take 3 GN iterations. The sketch parameters are $\ell = \ell_1 = 15$ and $\ell_2 = 31$. The RandSVD and Nystr\"om methods need 3 batches of $\ell = 15$ TLM and ADJ evaluations, and the $\ell = 15$ evaluations for each batch are done in a completely parallel way (but the ADJ evaluations must be done after the TLM). SingleView allows each batch of TLM and ADJ evaluations to be done simultaneously in parallel at an extra cost of $\ell_2 - \ell_1$ ADJ evaluations.}
    \label{tab:expt1_comp}
\end{table}
\Cref{fig:expt1_rmse} shows that all the variants of the GN method have the same relative analysis trajectory error, which is lower than the forecast error (whose trajectory is obtained by evolving the background without any assimilation).  
The relative error difference between the analysis trajectory and background becomes smaller over time, due to the diffusive nature of the chosen setting.
\Cref{fig:expt1_ncg} shows the total number of PCG iterations for different approaches. 
\Cref{tab:expt1_comp} provides a more quantitative picture of the full solution cost.
Concerning SketchPrec, the different preconditioned approaches require fewer PCG iterations compared with that of Prec\_$\B{\gampr}$ solution. 
For this small-scale problem, there aren't too many computational gains from using the Lanczos preconditioner, whereas the randomized methods allow for parallelism in the TLM and ADJ computations. Concerning the SketchSolv approach, the gains are again with the parallelization of the randomized methods. The randomized methods RandSVD and Nystr\"om sketches have similar behavior for both SketchSolv and SketchPrec, while SingleView is slightly more expensive to construct and requires more iterations, but at the same time, has a higher potential for parallelism (which is consistent with the theory). While not presented here, we also observed that the randomized methods are robust to the choice of random vectors (generated from different random seeds).

\subsubsection{Experiment 1.2: Condition Numbers and PCG iterations}

\begin{figure}[!ht]
    \centering
    \begin{subfigure}[b]{0.49\textwidth}
        \centering
        \includegraphics[width=\textwidth]{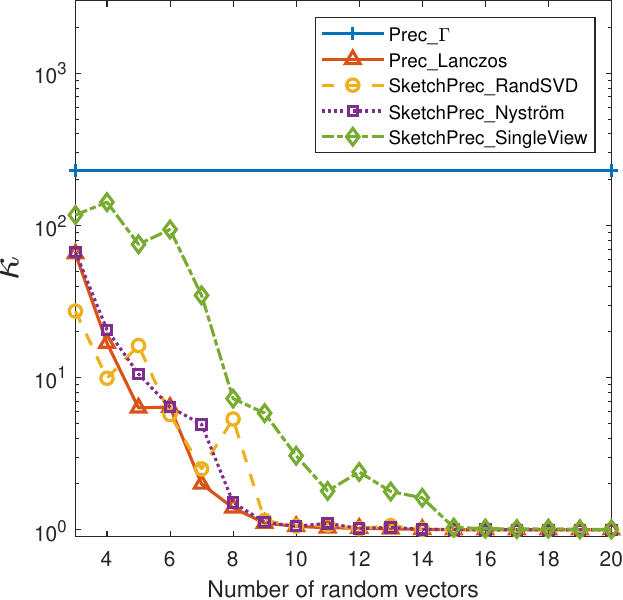}
        \caption{Condition number of $\MB{H}^{(0)}$. }
        \label{fig:expt1_cond}
    \end{subfigure}
    \hfill
    \begin{subfigure}[b]{0.49\textwidth}
        \centering
        \includegraphics[width=0.98\textwidth]{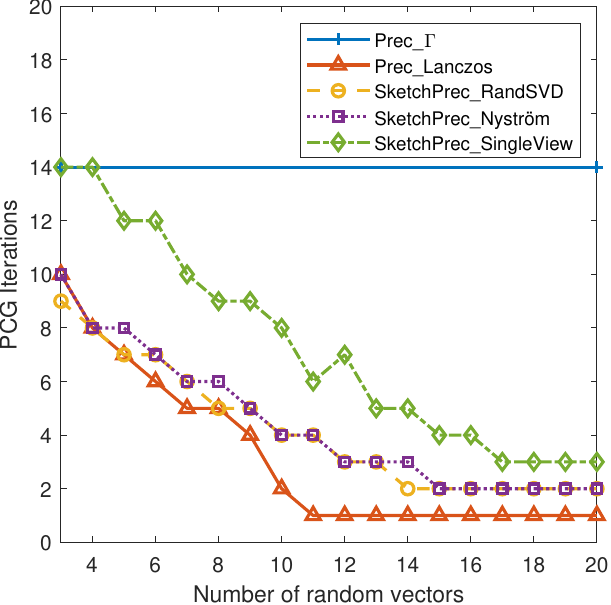}
        \caption{Total number of PCG Iterations}
        \label{fig:expt1_cg}
    \end{subfigure}
    \caption{Influence of the sketch size on the preconditioning of GN Hessian on the first GN iteration. The sketch size here refers to $\ell$ (RandSVD and Nystr\"om) and $\ell_1$ (SingleView).}
    \label{fig:expt1_randveccond}
\end{figure}
We retain the problem settings of Experiment 1, varying only the sketch sizes $\ell$ (Lanczos, RandSVD, Nystr\"om) and $\ell_1$ (SingleView), with $\ell_2$ constrained as $\ell_2 = 2\ell_1 + 1$ on the conditioning of {$(\B{I} + \Bh{H}^{(0)})^{-1}(\B{I} + \B{H}^{(0)}(\B{x}_0^{(0)}))$ (at the first iterate of the GN Hessian as in \cref{eqn:gniterf1}, where $\B{x}_0^{(0)} = \B{x}_0^{b} $)  and the total number of PCG iterations to solve the system $(\B{I} + \B{H}^{(0)})(\B{x}_0^{(0)}) \delta\B{x}^{(0)} = - \B{\gampr}^{1/2}\B{g}^{(0)}(\B{x}_0^{(0)})$.}

In \Cref{fig:expt1_cond}, we observe that a larger sketch size results in a better condition number and that the condition number is much lower than without a preconditioner. 
As expected from the theory in Section~\ref{sec:analysis}, using a larger sketch size to construct the preconditioners improves the system's conditioning.
In \Cref{fig:expt1_cg}, we see that a larger sketch size results in fewer PCG iterations, and the randomized preconditioners improve on the first-level preconditioner.
Between the different sketching methods, RandSVD and Nystr\"om require slightly smaller sketch sizes, and ultimately fewer TLMs/ADJs, than the SingleView method for convergence.
The Lanczos sketches make PCG converge in 1 step as they approximate the Krylov subspace of the system.

\subsubsection{Experiment 1.3: Conditioning of the GN Hessian}
\begin{figure}[!ht]
    \centering
    \begin{subfigure}[b]{0.49\textwidth}
    \centering
    \includegraphics[width=\linewidth]{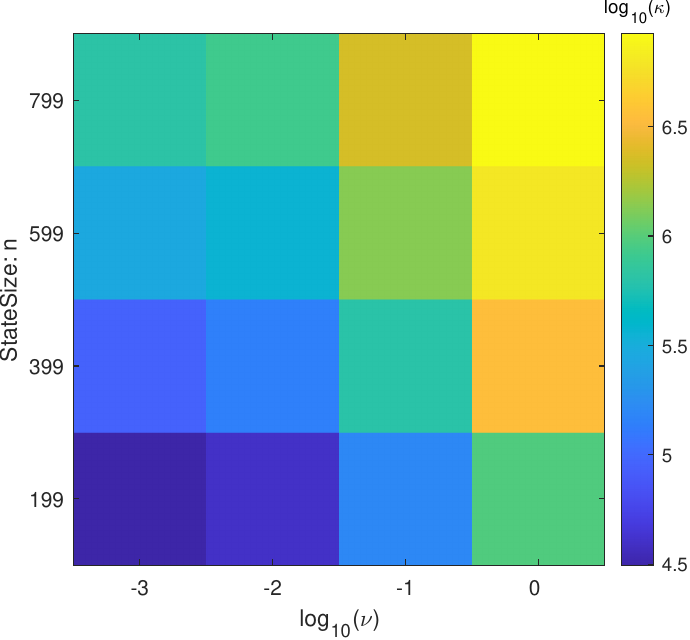}
    \caption{Log-condition number of $\MB{H}^{(0)}(\B{x}_0^{(0)})$.}
    \label{fig:expt1_condnv}    
    \end{subfigure}
    \begin{subfigure}[b]{0.49\textwidth}
    \centering
    \includegraphics[width=\linewidth]{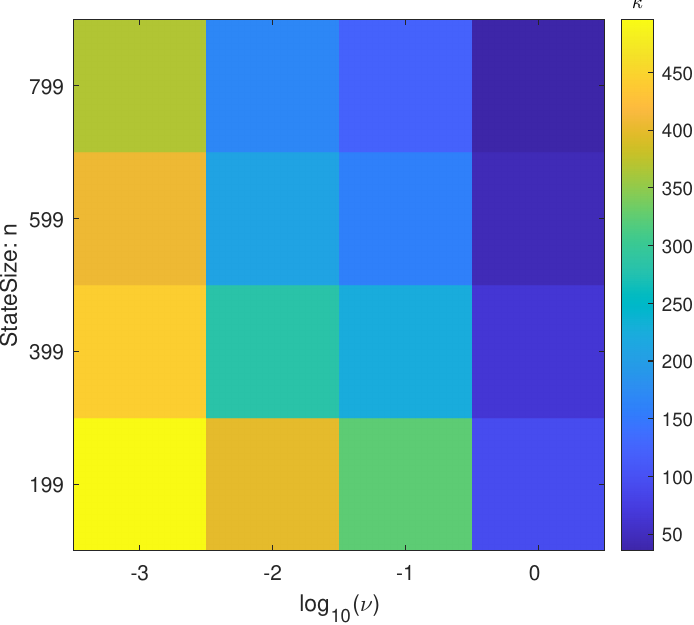}
    \caption{Condition number $\B{I} + \B{H}^{(0)}(\B{x}_0^{(0)})$.}
    \label{fig:expt1_condnv2}    
    \end{subfigure}
    \hfill
    \par
    \bigskip
    \begin{subfigure}[b]{0.49\textwidth}
        \centering
        \includegraphics[width=\linewidth]{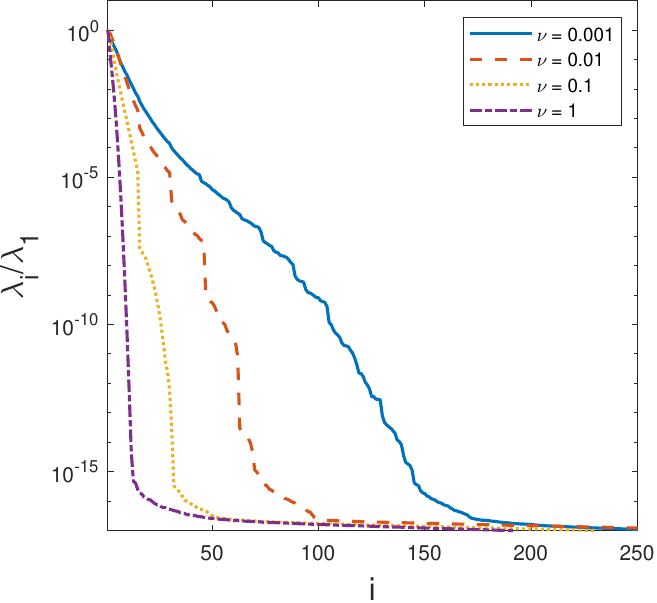}
        \caption{Relative spectral decay of $\B{H}^{(0)}(\B{x}_0^{(0)})$ for $n = 399$ }
        \label{fig:expt1_spec1}
    \end{subfigure}
    \hfill
    \begin{subfigure}[b]{0.49\textwidth}
        \centering
        \includegraphics[width=\linewidth]{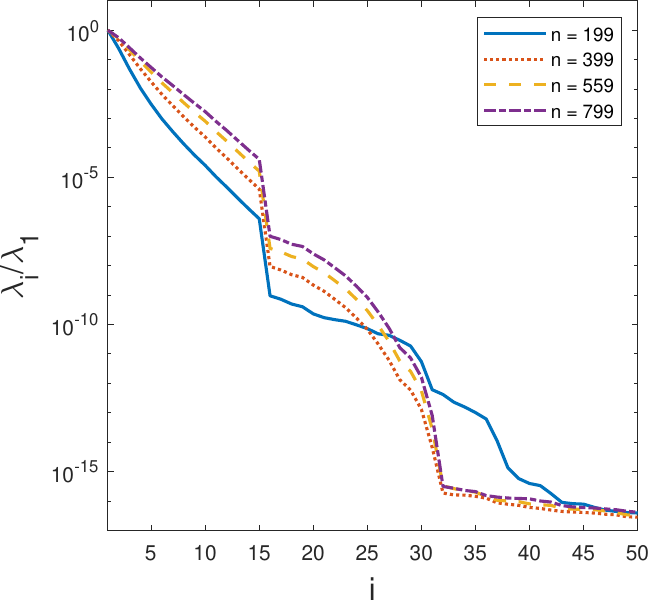}
        \caption{Relative spectral decay of $\B{H}^{(0)}(\B{x}_0^{(0)})$ for $\nu = 0.1$}
        \label{fig:expt1_spec2}
    \end{subfigure}
    \caption{Conditioning and spectral decay for different settings of the Burgers equation. Here, $i$ indexes $\sigma_i$ which is the $i$th largest eigenvalue.}
    \label{fig:expt1_condspec}
\end{figure}

{We now look at the condition number of $\MB{H}^{(0)}(\B{x}_0^{(0)})$ and $\B{I} + \B{H}^{(0)}(\B{x}_0^{(0)})$, by varying the discretization or state size ($n$) and diffusion coefficient ($\nu$).} 
To ensure consistency across the experiments, the background state $\B{x}_0^b$ is first computed for the finest grid $n = 799$, and its values are interpolated for the other grid sizes $n \in \{199, 399, 599\}$.
The background covariance chosen as in the previous experiments produces initial conditions that are smooth enough to behave similarly for all these discretizations. 
The observations are collected at the same spatial locations at all discretization levels (matching the observation x-locations in Figure~\ref{fig:expt1-example}) and do not change in value across the experiments.
{\Cref{fig:expt1_condnv,fig:expt1_condnv2} show the dependence of the condition number of $\MB{H}^{(0)}(\B{x}_0^{(0)})$ on the logarithm of diffusion coefficient $\nu$ and state size $n$. 
The dependence of the scaled eigenvalues of the regularized data misfit Hessian $\B{H}^{(0)}(\B{x}_0^{(0)})$ on $n$ and $\nu$ is shown in \Cref{fig:expt1_spec1}. }
We make {a few} observations here:
\begin{enumerate}
    \item The eigenvalues show a sharper decay with increased diffusion as seen in \Cref{fig:expt1_spec1}, and the condition number increases with a smaller diffusion coefficient $\nu$ (less diffusive) as seen from the rows of \Cref{fig:expt1_condnv}.
    \item The eigenvalue decay is similar across multiple values of $n$ as seen in \Cref{fig:expt1_spec2}, and the condition number increases with larger grid sizes as seen from the columns of \Cref{fig:expt1_condnv}.
    \item Experimentally, we see that the eigenvalue decay for systems with $\nu = 0.01, 0.1, 1$ in \Cref{fig:expt1_condspec} is consistent with the spectra of severely ill-posed problems (cf. Section~\ref{ssec:analysis_sum}). 
\end{enumerate}

\begin{figure}[!ht]
    \centering
    \includegraphics[width=\linewidth]{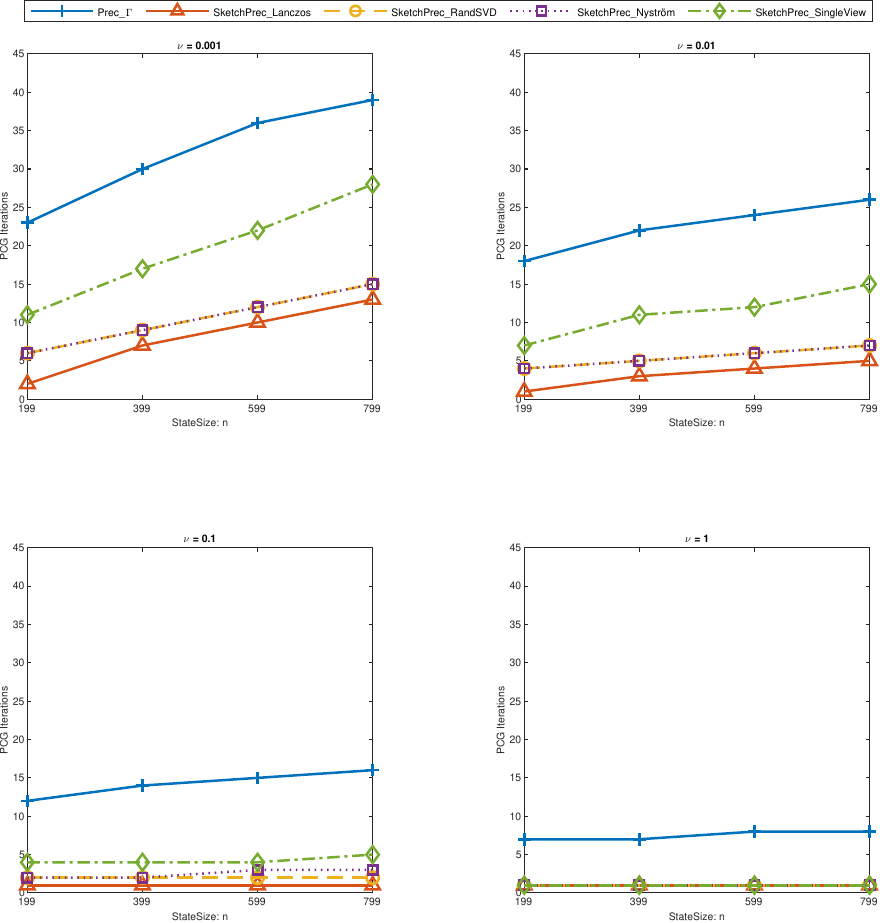}
    \caption{Effect of mesh refinement}
    \label{fig:expt1_gi}
\end{figure}

\subsubsection{Experiment 1.4: Effect of mesh refinement}
In this experiment, we use the settings of the previous experiment with sketch sizes $\ell = \ell_1 = 15$ and $\ell_2= 2\ell_1 +1$.
We report the number of PCG iterations using different methods to solve for the first descent direction in \Cref{fig:expt1_gi}.
For larger values of the diffusion coefficient $\nu$ $(\nu$ in the range $0.1-1$), no approach shows a significant change in the number of PCG iterations at different grid sizes. 
However, as the diffusion $\nu$ decreases ($\nu$ in the range $0.001-0.01$), the PDE becomes less diffusive, and the number of PCG iterations for all the methods shows a mild increase with increasing grid sizes. 
Therefore, the preconditioners only have grid-independent behavior for more diffusive regimes and show grid-dependent behavior for less diffusive regimes. 
In all the cases considered, the sketched preconditioners result in fewer PCG iterations compared to {using only} the background covariance as the preconditioner for the same grid size.

\subsection{Application 2: Barotropic vorticity equation}
\label{ssec:bve}
For the second application, we consider the barotropic vorticity equation, which represents a large-scale, turbulent, wind-driven circulation in a shallow oceanic basin~\cite[]{san2011qg}, that is a typical test problem for data assimilation~\cite[]{Talagrand_1987,evensen1994sequential,vanLeeuwen_2015,popov2021multifidelity}.
We follow the non-dimensional formulation as described in~\cite{san2011qg}:
\begin{equation}
\begin{split}
    \frac{\partial \omega}{\partial t} + \left( \frac{\partial \psi}{\partial y} \frac{\partial \omega}{\partial x} - \frac{\partial \psi}{\partial x} \frac{\partial \omega}{\partial y} \right) - \mathrm{Ro}^{-1} \frac{\partial \psi}{\partial x} &= \mathrm{Re}^{-1} \left( \frac{\partial^2 \omega}{\partial x^2} + \frac{\partial^2 \omega}{\partial y^2} \right) + \mathrm{Ro}^{-1} F,\\
    \frac{\partial^2 \psi}{\partial x^2} + \frac{\partial^2 \psi}{\partial y^2} &= - \omega,
\end{split}
\end{equation}
where $(x, y) \in \Omega = (0, 1)\times(-1, 1)$,  $\omega$ is the vorticity, $\psi$ is the stream function, $\mathrm{Re} = 200$ is the Reynolds number, $\mathrm{Ro} = 0.0016$ is the Rossby number, and a symmetric double gyre forcing function $F(x,y) = \sin(\pi y)$ and homogeneous Dirichlet boundary conditions (i.e., $\psi|_{\Omega} = \omega|_{\Omega} = 0$).

The PDE is spatially discretized on a uniform mesh using second-order central finite differences. We have the number of grid points $n = 128 \times 257 = 32,896$ with a grid spacing $\Delta x = \Delta y = 0.0078$.
A 3rd-order, explicit, total variation diminishing Runge Kutta~\cite[]{gottlieb1998tvd} scheme is used to evolve the semidiscretized PDE in time with a timestep of $\Delta t_{ts} = 0.000245$ time units.
The discrete Runge Kutta TLM and ADJ~\cite[]{Sandu_PA2006a} are used to compute the gradients and Hessian vector products. 
\begin{figure}[!ht]
    \centering
    \includegraphics[scale=0.65]{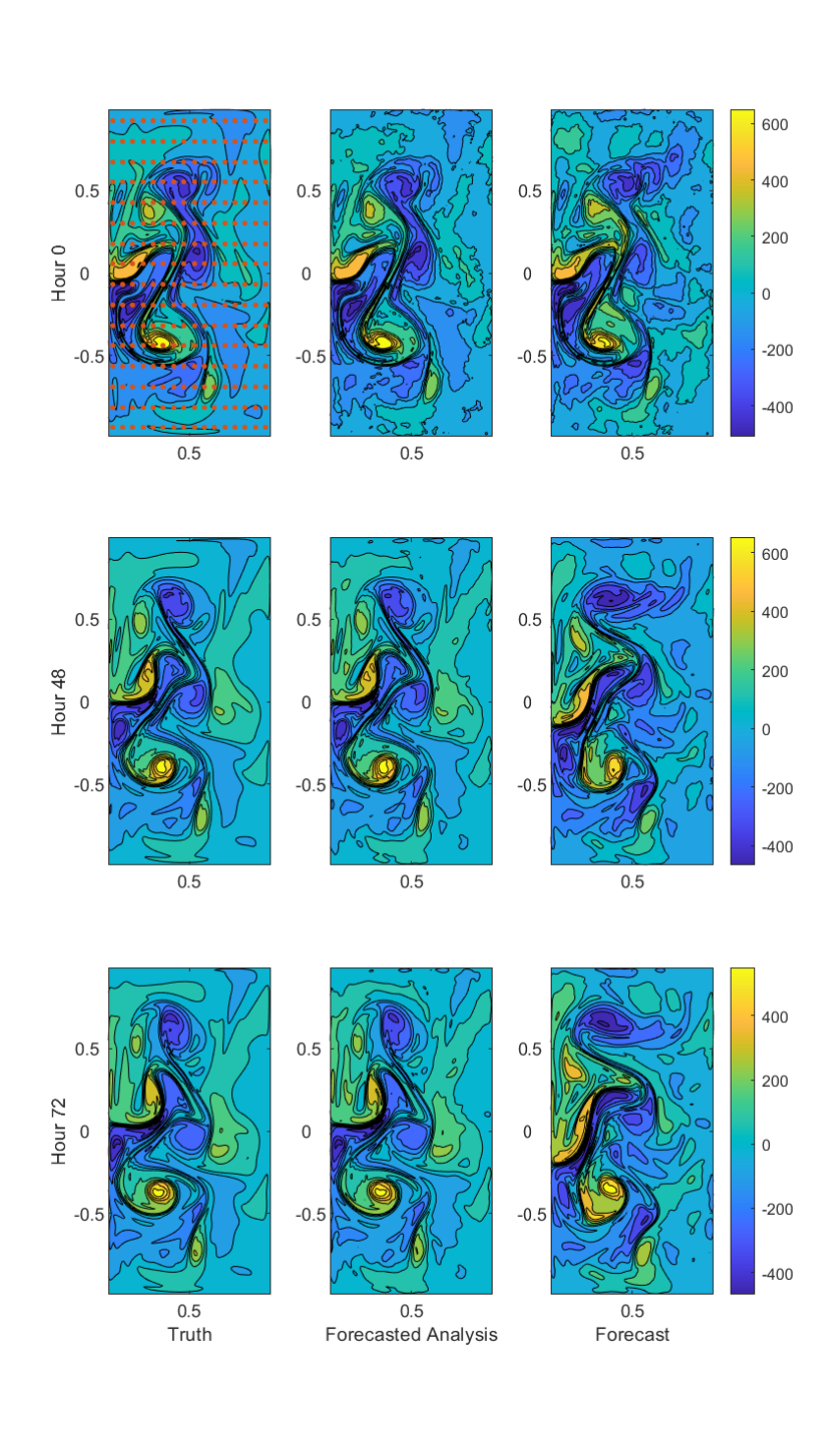}
    \caption{Time snapshots of the true vorticity, forecast, and the forecasted analysis. The assimilation cost considers the data from hour 6 to hour 48. The true data at hour 0 also shows the location of the observation sensors.}
    \label{fig:expt2-snap}
\end{figure}

\paragraph{Data and Background} The true initial condition is chosen after evolving a random, isotropic initial field through the system until a consistent set of gyres is obtained; see Figure~\ref{fig:expt2-snap}.
As standard in SC-4DVAR literature~\cite{evensen2022data,freitag2018low,Bousserez2020enhanced,dauvzickaite2021randomised}, the background estimate is chosen as $\B{x}^b_0 = \B{x}^{\rm true}_0 + \B\gampr^{\frac{1}{2}} \B\xi$, where the covariance matrix $\B\gampr = (-0.06 \B\Delta^*)^{-2}$, and $\B\Delta^*$ is the discrete 2-D Laplacian with homogeneous Dirichlet boundary conditions, and $\B{\xi} \sim \mc{N}(\B{0}, \B{I})$. %
The system is observed every $\Delta t = 0.001225$ non-dimensional time unit (corresponding to roughly $6$ hours of oceanic flow) for $n_t = 8$ sequential instances. 
The data is collected on a sensor network $16 \times 16$ grid, equally spaced in each dimension, as displayed in \cref{fig:expt2-snap}. 
The observation error covariance is $\B{R}_i = 361\cdot\B{I}_{256}$ for $1 \leq i \leq n_t$ (the parameter is chosen to be around $1$ \% of the maximum auto-variance.); similarly, the data is corrupted with the observational error drawn from $\B\epsilon_i^{\rm obs} \sim \mc{N}(\B{0},\B{R}_i)$.

\paragraph{Optimization and PCG settings} 
{While both SketchPrec and SketchSolv methods were tested, SketchSolv failed to converge within a reasonable amount of wall-time (likely due to the inexactness in the descent direction) and its results are not reported.}
As before, the optimization is initialized with the background $\B{x}_0^{(0)} = \B{x}^b_0$ and terminated when $\frac{\|\nabla \mc{J}^{(k)} \|_\infty}{\|\nabla \mc{J}^{(0)} \|_\infty} < 10^{-6}$. 
PCG is terminated when the relative residual of the $k$-th iterate achieves a prespecified tolerance (here, $10^{-9}$). 
We compare the methods
\begin{enumerate*}[label=(\roman*)]
    \item basic : Prec\_$\B{\gampr}$, Prec\_Lanczos, PrecA\_Lanczos, 
    \item SketchPrec : RandSVD, Nystr\"om, and SingleView,
    \item SketchPrecA : RandSVD, Nystr\"om, and SingleView.
\end{enumerate*}
For the fixed sketch size case, we take $\ell = \ell_1 = 384$, and for the adaptive case, we take $\epsilon_{\rm re} = 10$ and $\epsilon_{\rm sk} = 1.01$.

PrecA\_Lanczos refers to the adaptive method (in both sketch size and decision). The adaptivity in sketch size requires terminating the Lanczos procedure when {the criterion  $\kappa_{\rm sk} \le \epsilon_{\rm sk}$ (see, \cref{eqn:condtext}) is met.} It is similar to \cref{alg:arandsvd} with an initial $\ell$ Lanczos iterations, after which we test for the sketch size adaptivity condition. If it fails, we increment the Lanczos subspace by $\hat{\ell}$ more vectors and repeat till the condition is satisfied. 

\begin{figure}[!ht]
    \centering
    \begin{subfigure}[b]{0.49\textwidth}
        \centering
        \includegraphics[width=\linewidth]{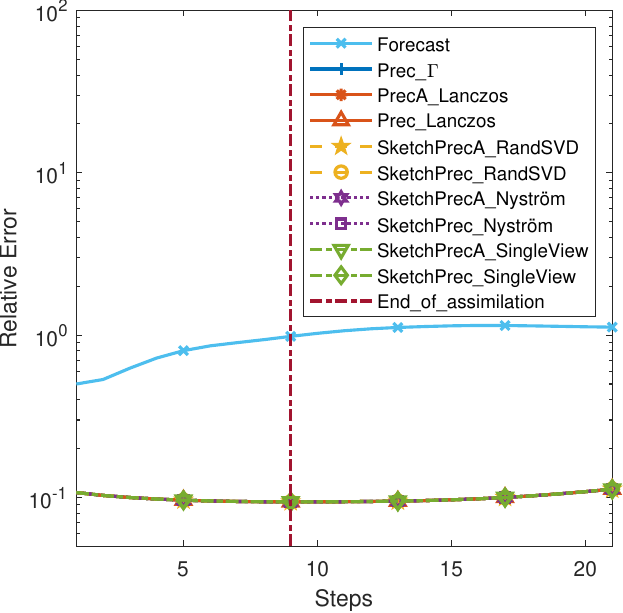}
        \caption{Relative Error for BVE.}
        \label{fig:expt2_relrmse}
    \end{subfigure}
    \hfill
    \begin{subfigure}[b]{0.49\textwidth}
        \centering
        \includegraphics[width=\linewidth]{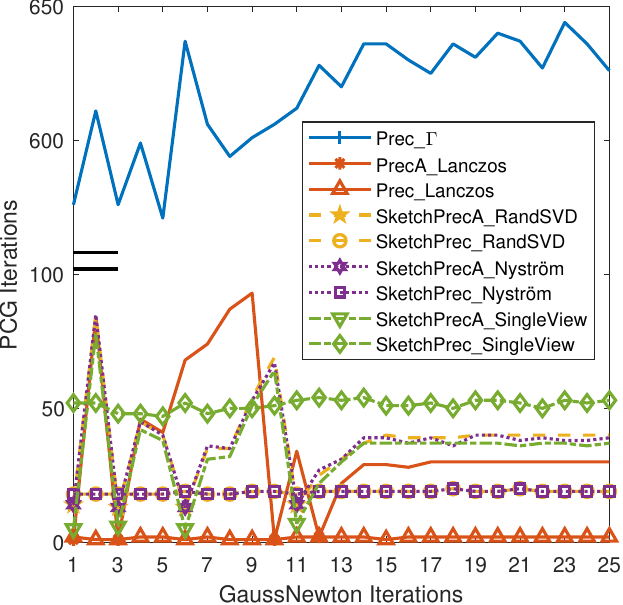}
        \caption{GN iterations vs PCG iterations.}
        \label{fig:expt2_ncg}
    \end{subfigure}
    \caption{Plot showing the relative error and PCG iterations. In \cref{fig:expt2_ncg}, the marker indicates the decision to sketch. { PrecA\_Lanczos sketchs at the GN iterations \{1, 3, 10, 12\}, and all SketchPrecA methods sketch at GN iterations\{1, 3, 6, 11\}.} Also, note the jump in the y-axis at 100 for \cref{fig:expt2_ncg}. }
    \label{fig:expt2}
\end{figure}  

\begin{figure}[!ht]    
    \centering
    \begin{subfigure}[b]{0.49\textwidth}
        \centering
        \includegraphics[width=\linewidth]{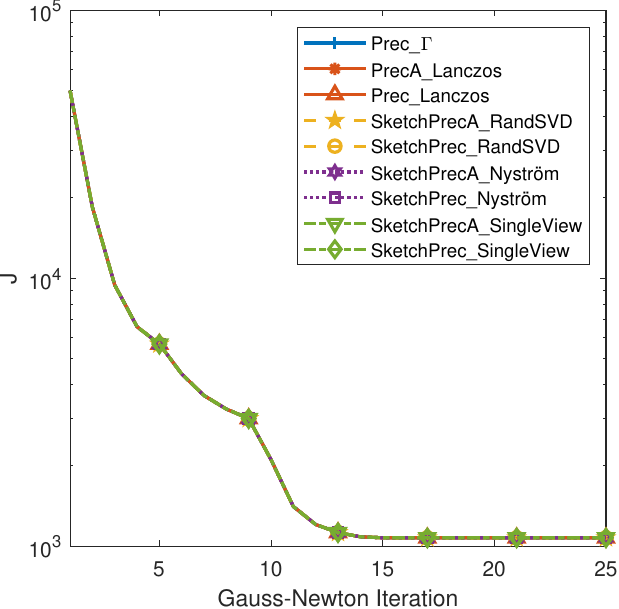}
        \caption{The SC-4DVAR cost function.}
        \label{fig:expt2_cost}
    \end{subfigure}
    \hfill
    \begin{subfigure}[b]{0.49\textwidth}
        \centering
        \includegraphics[width=\linewidth]{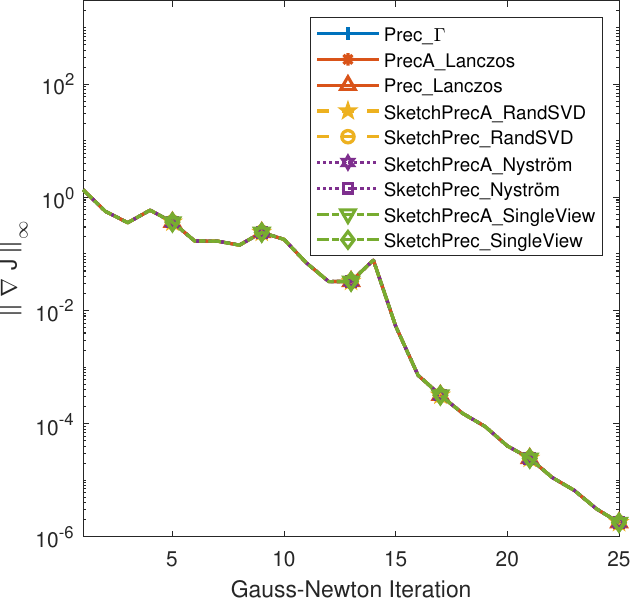}
        \caption{Infinity norm of the gradient.}
        \label{fig:expt2_grad}
    \end{subfigure}
    \caption{Plots showing the convergence of the optimization.}
    \label{fig:expt2b}
\end{figure}
\begin{table}[!ht]
    \centering
    \begin{tabular}{ c | c | c | c | c | c | c}
       Method & \#PCG & \multicolumn{2}{c|}{Par/Offline/Setup} & \multicolumn{3}{c}{Seq/Online} \\ \hline & \thead{Iterations} & \thead{\#TLM} & \thead{\#ADJ}  & \thead{\#FWD} & \thead{\#TLM} & \thead{\#ADJ} \\\hline
      \multicolumn{7}{c}{Basic} \\ \hline
         Prec\_$\B{\gampr}$ & 15441 & --- & --- & 57 & 15441 & 15466 \\
         Prec\_Lanczos & 43 & --- & --- & 57 & 9643 & 9668 \\
         PrecA\_Lanczos & 910 & --- & --- & 57 & 2958 & 2983  \\\hline
         \multicolumn{7}{c}{SketchPrec} \\\hline
         RandSVD & 469 & 9600 & 9600 & 57 & 469 & 494 \\
         Nystr\"om & 469 & 9600 & 9600 & 57 & 469 & 494 \\
         SingleView & 1282 & 9600 & 19225 & 57 & 1282 & 1307 \\\hline
         \multicolumn{7}{c}{SketchPrecA} \\\hline
         RandSVD & 947 & 2048 & 2048 & 57 & 947 & 973 \\
         Nystr\"om & 935 & 2048 & 2048 & 57 & 935 & 960  \\
         SingleView & 853 & 5888 & 11780 & 57 & 853 & 878 \\
    \end{tabular}
    \caption{Summary of the numerical results for the BVE. All experiments use $25$ outer (GN) iterations. The results for the adaptive methods (PrecA\_Lanczos, and all SketchPrecA) do not include the cost of computing $\kappa_{\rm sk}$ and $\kappa_{\rm re}$ as it is negligible. \Cref{tab:expt2_comp_sum} summarizes the total cost for each method.}
    \label{tab:expt2_comp}
\end{table}

\paragraph{Results}

\Cref{fig:expt2_relrmse} shows the relative error over the trajectory for the analyses, and all methods have very similar error trajectories, which are all lower than the forecast trajectory error.
To demonstrate the convergence of the optimization, we also show the decrease in cost function (in \cref{fig:expt2_cost}) and the first-order optimality criterion (in \cref{fig:expt2_grad}).
These figures show that the optimization has converged and the different methods behave identically.
\Cref{fig:expt2_ncg} compares the PCG iterations required for convergence at each GN step for each of the sketching methods.
The SketchPrec methods with fixed sketch size---Lanczos, RandSVD, Nystr\"om, and SingleView---all show a large decrease in the total number of PCG iterations compared to {Prec\_$\B{\gampr}$}. 
The fixed sketch size methods are relatively constant in terms of the number of PCG iterations per GN step. 
%
%

The adaptive methods (in \cref{fig:expt2_ncg}) recompute the sketch at 4 GN iterations. 
For the remaining GN iterations, the previously computed preconditioner is reused.
For the adaptive Lanczos, RandSVD, and Nystr\"om methods, the sketch size was adaptively determined to be $\ell = 512$ at every recomputing iteration.
The adaptive SingleView chose $\ell_1 = 1536, 1408, 1536, 1408$, with $\ell_2 = 2 \ell_1 + 1$ for the recomputing iterations.

\Cref{tab:expt2_comp} shows the computational cost for all the different preconditioning methods.
{The cost of a method is measured by the total number of FWD, TLM, and ADJ model evaluations. Typically, ADJ is the most expensive computation among the three, but for this test problem, it is reasonable to assume that all three -- FWD, TLM, and ADJ cost the same. Under this assumption we have the following table (\Cref{tab:expt2_comp_sum}) with the costs in terms of PDE solves:}

\begin{table}[!ht]
    \centering
    \begin{tabular}{ c | c | c  }
       Method & {Par/Offline/Setup} & {Seq/Online} \\ \hline  
         Prec\_$\B{\gampr}$  & ---  & 30946 \\ 
         Prec\_Lanczos & ---  & 19368 \\ 
         PrecA\_Lanczos &  ---  & 5998   \\ \hline
         \multicolumn{3}{c}{SketchPrec}  \\ \hline
         RandSVD & 19200 & 1020 \\ 
         Nystr\"om & 19200 & 1020 \\ 
         SingleView & 28225 & 2646 \\ \hline
         \multicolumn{3}{c}{SketchPrecA}  \\ \hline
         
         RandSVD & 4096 & 1950 \\
         Nystr\"om & 4096 & 1950  \\
         SingleView & 17668 &  1800 \\
    \end{tabular}
    \caption{Summary of total costs in terms of PDE solves. Notice that all the offline costs can be parallelized using multiple embarrassingly parallel batches. Note that this table summarizes the results from \Cref{tab:expt2_comp}.}
    \label{tab:expt2_comp_sum}
\end{table}

Observe that the SketchPrec methods have similar total costs as Prec\_Lanczos, and SketchPrecA methods have similar total costs as PrecA\_Lanczos, all of which have a total lower cost than Prec\_$\B{\gampr}$. We emphasize that all the randomized variants allow for multiple embarrassingly parallel batches of PDE solves when compared to Prec\_$\B{\gampr}$, Prec\_Lanczos, and PrecA\_Lanczos.
For all methods---Lanczos, RandSVD, Nystr\"om, SingleView---the adaptive preconditioning methods result in a much lower cost than the respective fixed-size (preconditioner computed at every GN iteration) counterparts. 
The SketchPrecA methods have a much lower total cost than the SketchPrec counterparts, at the cost of doing some extra sequential TLM and ADJ computations. 
To conclude, in an ideal scenario with many compute cores and a lot of memory, one would prefer the SketchPrec methods, but when there is contention for compute resources and memory, SketchPrecA must be used. 
If parallel processing is not an option, then the randomized methods (RandSVD and Nyst\"om) are only as good as the Lanczos method, and either can be used.

\section{Conclusion and future work}

In this paper, we propose new algorithms for solving and preconditioning the SC-4DVAR problem to accelerate the time to solution and reduce total computations.
The algorithms are based on different randomized sketching techniques: three methods---RandSVD, Nystr\"om, and SingleView---with fixed sketch size and adaptive approaches to determine when to update the preconditioner and the sketch size. 
We perform structural and expectational analyses of the proposed methods and demonstrate their applicability to two model problems, namely the 1D-Burgers equation and the 2D barotropic vorticity equation. 
We demonstrate the gains from our preconditioning method in three ways:
\begin{enumerate*}[label={(\roman*)}]
    \item a substantial reduction in the number of TLM and ADJ solves in the online phase,
    \item a substantial reduction in the number of PCG iterations,
    \item parallelization of a major portion of the TLM and ADJ solves {in the offline phase}.
\end{enumerate*}

{For practitioners, we recommend using the SketchPrecA methods---which is the adaptive version of the randomized preconditioners. This is the most practical version since it allows for adaptivity in both the sketch size and the use of sketching across multiple Gauss-Newton iterations. Between the three sketching techniques, RandSVD and Nystr\"om have comparable performance, and SingleView is more beneficial in a parallel computational setting. The randomized preconditioned methods are especially effective when the data-misfit portion of the Hessian is approximately low-rank.}
\paragraph{Future Work}
In \cite{Tshimanga_2008}, a class of limited memory preconditioners are built by using the conjugate vectors generated during the CG iterations. It will be interesting to pursue this idea for adaptive randomized preconditioners, similar to PrecA\_Lanczos. 
In this work, we have demonstrated the approach on synthetic problems. Demonstrating the effectiveness of these methods on large-scale DA problems, such as those encountered in operational NWP, is a direction that we wish to pursue. These large-scale problems might not necessarily share the properties of the test problems on which the approaches have been demonstrated. This will pose additional challenges. Another challenge in operational NWP is that the models might not be equipped with first-order adjoints, which is necessary for our methods. However, the emergence of data-driven NWP surrogates such as FourCastNet, GraphCast, and Pangu weather, coupled with automatic differentiation, can potentially alleviate this issue \cite{pathak2022fourcastnet, lam2023learning, bi2022pangu}. Moreover, randomized approaches to data assimilation have mainly been explored for variational data assimilation and have not been effectively used in ensemble data assimilation so far. 
It would also be interesting to explore randomized preconditioners in the context of hybrid ensemble-variational data assimilation methods that are a popular area of research~\cite{subrahmanya2023evfp}.
{These ideas could also be applied to solving general inverse problems such as subsurface flow inversion, and contaminant source detection.}

\begin{acknowledgements}
    The authors thank Adrian Sandu, Ahmed Attia, and Srinivas Eswar for helpful discussions.
\end{acknowledgements}

\begin{decla}{Author Contributions}
    A.N.S wrote the required code, verified and validated the code, performed multiple experiments, and created figures. A.K.S and V.R. conceptualized the formulation, and analyzed the formulation. All authors contributed equally to writing and editing the manuscript.
\end{decla}

\begin{decla}{Funding}
    AKS was supported, in part, by the National Science Foundation through the grant DMS-1845406 and the Department of Energy through the grant DE-SC0023188. Vishwas Rao is supported by the U.S. Department of Energy, Office of Science, Advanced Scientific Computing Research Program under contract DE-AC02-06CH11357.
\end{decla}

\begin{decla}{Data Availability}
    This manuscript does not report data generation or analysis. The code used will be made available on request.
\end{decla}

\section*{Declarations}

\begin{decla}{Competing Interests}
    The authors declare that they have no competing interests.
\end{decla}

\begin{decla}{Ethics Approval and Consent to Participate}
    Not applicable.
\end{decla}

\begin{decla}{Consent for Publication}
    Not applicable.
\end{decla}

%

\bibliographystyle{spbasic}      
\bibliography{references, references_da, references_randnla}   

\end{document}

%% file: notation.tex
\newcommand{\mb}[1]{\mathbb{#1}}
\newcommand{\B}[1]{\boldsymbol{#1}}
\newcommand{\MB}[1]{\boldsymbol{\mathcal{#1}}}
\newcommand{\MBh}[1]{\widehat{\boldsymbol{\mathcal{#1}}}}
\newcommand{\Bh}[1]{\widehat{\boldsymbol{#1}}}
\newcommand{\Bt}[1]{\widetilde{\boldsymbol{#1}}}
\newcommand{\R}{\mathbb{R}}
\renewcommand{\t}{{}^\top}
\newcommand{\mc}[1]{\mathcal{#1}}
\newcommand{\expect}[1]{\mathbb{E}\left[#1\right]}

\newcommand{\range}{\mathcal{R}}
\newcommand{\id}{\mathsf{id}}
\newcommand{\bmat}[1]{\begin{bmatrix}#1 \end{bmatrix}}

\newcommand{\gampr}{\Gamma}

\DeclareMathOperator*{\argmin}{\arg\min}

\newenvironment{decla}[1]{
    \par\addvspace{10pt}\small\rmfamily
    \trivlist
    \item[\hskip\labelsep
    {\bfseries #1}] }
    {\endtrivlist\addvspace{0pt}}